\documentclass[11pt,letterpaper]{article}

\usepackage{amsmath}                    
\usepackage{amssymb}                    
\usepackage{amsthm}                     
\usepackage{amsfonts}                   

\usepackage{caption}
\usepackage{subcaption}
\usepackage{enumerate}
\usepackage[nospace,noadjust]{cite}

\usepackage{pgfplots}
\usetikzlibrary{arrows}
\pgfplotsset{compat=newest}
\tikzset{My Line Style/.style={smooth, samples=100,color=black}}

\usepackage[margin=1in]{geometry}

\theoremstyle{definition}
\newtheorem{dfn}{Definition}[section]
\theoremstyle{plain}
\newtheorem{thm}[dfn]{Theorem}

\newtheorem{lemma}[dfn]{Lemma}
\newtheorem{cor}[dfn]{Corollary}
\newtheorem{fact}[dfn]{Fact}
\newtheorem{claim}[dfn]{Claim}

\DeclareMathOperator*{\E}{{\rm E}}
\DeclareMathOperator*{\R}{\mathbb{R}}
\DeclareMathOperator*{\Z}{\mathbb{Z}}
\DeclareMathOperator*{\N}{\mathbb{N}}

\usepackage{bbm}
\DeclareMathOperator*{\1}{\mathbbm{1}}

\DeclareMathOperator*{\sgn}{{\rm sgn}}

\def\ls{\lambda_s}
\def\lS{\lambda_S}
\def\Otilde{\tilde O}
\def\poly{{\rm poly}}

\def\taumix{\tau_{\rm mix}}
\def\Tcoup{T_{\rm coup}}

\def\tGln{\widetilde{G}_{d_n}}
\def\Gln{G_{d_n}}

\def\gc{\beta}
\def\vf{\theta}

\def\gcrc{\theta_r}

\def\ts{\Theta_{\rm s}}
\def\tS{\Theta_{\rm S}}

\title{Dynamics for the mean-field random-cluster model}
\author{Antonio Blanca\thanks{Computer Science Division, U.C. Berkeley, Berkeley, CA 94720. Email: {\tt ablanca@cs.berkeley.edu}.  Research supported in part by a U.C. Berkeley Chancellor's Fellowship for Graduate Study, an NSF Graduate Research Fellowship and NSF grants CCF-1016896 and CCF-1420934.}
\and Alistair Sinclair\thanks{Computer Science Division, U.C. Berkeley, Berkeley, CA 94720. Email: {\tt sinclair@cs.berkeley.edu}.  Research supported in part by NSF grants CCF-1016896 and CCF-1420934.}
}

\begin{document}

\date{}

\maketitle

\thispagestyle{empty}

\begin{abstract}
\noindent
The random-cluster model has been widely studied as a unifying framework for random graphs, spin systems and random spanning trees, but its dynamics have so far largely resisted analysis.
In this paper we study a natural
non-local Markov chain known as the Chayes-Machta dynamics for the mean-field case of the
random-cluster model, and identify a critical regime $(\lambda_s,\lambda_S)$ of the model
parameter~$\lambda$ in which the dynamics undergoes an exponential slowdown. Namely, we prove that the mixing time is $\Theta(\log n)$ if $\lambda \not\in [\ls,\lS]$, and $\exp(\Omega(\sqrt{n}))$ when $\lambda \in (\ls,\lS)$. These results hold for all values of the second model parameter~$q > 1$. In addition, we prove that the local heat-bath dynamics undergoes a similar exponential slowdown in $(\ls,\lS)$.




\end{abstract}

\newpage
\setcounter{page}{1}

\pagebreak
\section{Introduction}

{\bf Background and previous work.}\ \ Let $H=(V,E)$ be a finite graph.  The {\it random-cluster model\/} on~$H$ with parameters
$p\in(0,1)$ and $q>0$ assigns to each subgraph $(V,A \subseteq E)$ a probability
\begin{equation}\label{eq:rcmeasure}
\mu_{p,q}(A) \propto p^{|A|}(1-p)^{|E|-|A|} q^{c(A)},\nonumber
\end{equation}
where $c(A)$ is the number of connected components in~$(V,A)$. $A$ is a configuration of the model.


The random-cluster model was introduced in the late 1960s by Fortuin and Kasteleyn~\cite{FK}
as a unifying framework for studying random graphs, spin systems in physics and random spanning trees; see the book~\cite{Grimmett} for extensive background. When $q=1$ this model corresponds to the standard Erd\H{o}s-R\'enyi model on subgraphs of~$H$,
but when $q>1$ (resp., $q<1$) the resulting probability measure favors subgraphs with more (resp., fewer) connected components, and is thus a strict generalization. 


For the special case of integer $q \ge 2$ the random-cluster model is, in a precise sense, dual to the classical ferromagnetic {\it $q$-state Potts model}, where configurations are assignments of spin values $\{1,\ldots,q\}$ to the vertices of~$H$; the duality is established via a coupling of the models (see, e.g., \cite{ES}). Consequently, the random-cluster model illuminates much of the physical theory of the Ising/Potts models. Indeed, recent breakthrough work by Beffara and Duminil-Copin \cite{BDC} uses the geometry of the random-cluster model in $\Z^2$ to establish the critical temperature of the $q$-state Potts model, settling a long-standing conjecture.

At the other extreme, when $q, p \rightarrow 0$ and $p$ approaches zero at a slower rate (i.e., $q/p \rightarrow 0$) the random-cluster measure $\mu_{p,q}$ converges to the uniform random spanning tree measure on $H$. Random spanning trees are fundamental probabilistic objects, whose relevance goes back to Kirchhoff's work on electrical networks \cite{Kir}. 




In this paper we investigate the dynamics of the random-cluster model, i.e.,
Markov chains on random-cluster configurations that are reversible w.r.t.\ $\mu_{p,q}$ and thus converge to it. The dynamics
of physical models are of fundamental interest, both as evolutionary processes in
their own right and as Markov chain Monte Carlo (MCMC) algorithms for sampling configurations in equilibrium. In both
these contexts the central object of study is the {\it mixing time}, i.e., the number of
steps until the dynamics is close to the equilibrium measure~$\mu_{p,q}$ starting
from any initial configuration. While 
dynamics for the Ising and Potts models have been widely studied, very
little is known about random-cluster dynamics. The main reason for this appears to be the fact that connectivity is a global property which has led to the failure of existing Markov chains analysis tools.

We focus on the {\it mean-field\/} case, where $H$ is the complete graph on $n$ vertices.
In this case the random-cluster model may be viewed as the standard random graph model ${\cal G}_{n,p}$,
enriched by a factor that depends on the component structure. As we shall see, the mean-field
case is already quite non-trivial; moreover, it has historically proven to be a useful starting
point in understanding the dynamics on more general graphs.  The structural properties of the mean-field
model are already well understood~\cite{BGJ,LL}; in particular, it exhibits a phase transition
(analogous to that in ${\cal G}_{n,p}$) corresponding to the appearance of a ``giant" component of linear size. It is natural here to re-parameterize by setting $p=\lambda/n$; the phase transition then
occurs at the critical value $\lambda = \lambda_c(q)$ given by 
\begin{equation}\label{eq:lambdacrit}
\lambda_c(q) = \begin{cases}
q & \text{for $0<q\le 2$};\\
2{\textstyle\left(\frac{q-1}{q-2}\right)\log(q-1)} & \text{for $q>2$}.
\end{cases} \nonumber
\end{equation}
For $\lambda < \lambda_c(q)$ all components are of size $O(\log n)$ w.h.p.\footnote{We say that an event occurs {\it with high probability (w.h.p.)\/} if it occurs with probability approaching 1 as $n \rightarrow \infty$.}, while for
$\lambda>\lambda_c(q)$ there is a unique giant component of size $\theta n$ (for some constant $\theta$ that depends on~$q$ and $\lambda$). The former regime is called the {\it disordered phase}, and the latter is the {\it ordered phase}. Henceforth we assume $q > 1$, since the
$q < 1$ regime is structurally quite different; the dynamics are trivial for $q=1$.

Our main object of study is a non-local dynamics known as the 
{\it Chayes-Machta (CM) dynamics\/}~\cite{CM}.  
Given a random-cluster configuration~$(V,A)$, one step of this dynamics is defined as follows:
\begin{itemize}
	\item[(i)]{\it activate\/} each connected component of~$(V,A)$ independently with probability $1/q$; 
	\item[(ii)] remove all edges connecting active vertices;
	\item[(iii)] add each edge connecting active vertices independently with probability $p,$ leaving the rest of the configuration unchanged.	
\end{itemize} 

\noindent
It is easy to check that this dynamics is reversible w.r.t.~$\mu_{p,q}$ \cite{CM}. Until now, the mixing time of the CM dynamics has not been rigorously established for any non-trivial random-cluster measure $\mu_{p,q}$ on any graph. Our goal in this paper is to analyze the CM dynamics in the mean-field case for all values of $q > 1$ and all values of $\lambda > 0$.

For integer~$q$, the CM dynamics is a close cousin of the well studied and widely used {\it Swendsen-Wang (SW) dynamics\/} \cite{SW}.
The SW dynamics is primarily a dynamics for the Ising/Potts model, but it may alternatively be viewed as a Markov chain for the random-cluster model using the coupling of these measures mentioned earlier. However, the SW dynamics is only well-defined for integer $q$, while the random-cluster model makes perfect sense for all $q>0$. The CM dynamics was introduced precisely in order to allow for this generalization.

The SW dynamics for the mean-field case is fully understood for $q=2$:
recent results of Long, Nachmias, Ning and Peres~\cite{LNNP}, building on earlier work of
Cooper, Dyer, Frieze and Rue~\cite{CDFR}, show that the mixing time is $\Theta(1)$
for $\lambda<\lambda_c$, $\Theta(\log n)$ for $\lambda>\lambda_c$, and $\Theta(n^{1/4})$
for $\lambda=\lambda_c$. Until recently, the picture for integer $q \ge 3$  was much less complete: Huber~\cite{Huber} gave bounds of $O(\log n)$ and
$O(n)$ on the mixing time when $\lambda$ is far below and far above~$\lambda_c$ respectively, while Gore
and Jerrum~\cite{GJ} showed that at the critical value $\lambda=\lambda_c$ the mixing time
is $\exp(\Omega(\sqrt{n}))$. All these results were developed for the Ising/Potts model, so their relevance to the random-cluster model is limited to the case of integer~$q$. In work that appeared after the submission of this manuscript \cite{BSmf}, Galanis, \v{S}tefankovi\v{c} and Vigoda \cite{GSV} provide a more comprehensive analysis of the $q \ge 3$ mean-field case. Finally, for the very different case of the $d$-dimensional torus, Borgs et al. \cite{BCT,BFKTVV} proved exponential lower bounds for the mixing time of the SW dynamics for $\lambda = \lambda_c$ and $q$ sufficiently large. 


%

Our work is the first to provide tight bounds for the mixing time of any random-cluster dynamics for general (non-integer) values of $q$. 

\par\medskip\noindent
{\bf Results.}\ \ To state our results we identify two further critical points,
$\ls(q)$ and $\lS(q)$, with the property that $\ls(q) \le \lambda_c(q) \le \lS(q)$. (For $1 < q \le 2$ these three points coincide; for $q>2$ they are all distinct.) The definitions
of these points are somewhat technical and can be found in Section~\ref{sec:preliminaries}.

Our first result shows that the CM dynamics reaches equilibrium very rapidly for~$\lambda$
outside the ``critical" window $[\ls,\lS]$. Moreover, our bounds are tight throughout the fast mixing regime.
\begin{thm}\label{thm:intro1}
	For any $q > 1$, the mixing time of the mean-field CM dynamics is $\Theta(\log n)$ for $\lambda \not\in [\ls,\lS]$.
\end{thm}
\noindent
Our next result shows that, {\it inside} the critical window $(\ls, \lS)$, the mixing time is
dramatically larger. 
(We state this result only for $q>2$ as otherwise the window is empty.)
\begin{thm}\label{thm:intro2}
	For any $q>2$, the mixing time of the mean-field CM dynamics is $e^{\Omega(\sqrt{n})}$
	for $\lambda\in (\ls,\lS)$.
\end{thm}
\noindent
We now provide an interpretation of the above results. When $q>2$ the mean-field random-cluster model exhibits a
{\it first-order\/} phase transition, which means that at criticality ($\lambda = \lambda_c$) the ordered and disordered phases
mentioned earlier {\it coexist}~\cite{LL}, i.e., each contributes about half of the probability mass. (For $q\le 2$, there is no phase coexistence.) Phase coexistence suggests exponentially slow mixing for most natural dynamics, because
of the difficulty of moving between the phases. Moreover, by continuity we should expect that, within a constant-width interval around~$\lambda_c$, the
effect of the non-dominant phase (ordered below $\lambda_c$, disordered above~$\lambda_c$) will still be felt, as it will form a
second mode (local maximum) for the random-cluster measure. This leads to so-called {\it metastable\/} states near that
local maximum from which it is very hard to escape, so slow mixing should persist throughout this interval. Intuitively, the values $\lambda_s,\lambda_S$
mark the points at which the local maxima disappear. A similar phenomenon was
captured in the case of the Potts model by Cuff {\it et al.}~\cite{CDLLPS}. Our results make the above picture for the dynamics rigorous for the random-cluster
model for all $q>2$; notably, in contrast to the Potts model, in the random-cluster model metastability affects the mixing time on {\it both\/} sides of~$\lambda_c$.  
Note that our results leave open the behavior of the mixing time exactly at $\ls$ and $\lS$.



As a byproduct of our main results above, we deduce new bounds on the mixing time of {\it local\/}
dynamics for the random-cluster model (i.e., dynamics that modify only a constant-size
region of the configuration at each step). For definiteness we consider the canonical
{\it heat-bath (HB) dynamics\/}, which in each step updates a single edge of the current
configuration $(V,A)$ as follows:
\begin{itemize}	
	\item[(i)] pick an edge $e\in E$ u.a.r;
	\item[(ii)] replace $A$ by $A\cup \{e\}$ with probability $\frac{\mu_{p,q}(A\cup\{e\})}{\mu_{p,q}(A\cup\{e\})+\mu_{p,q}(A\setminus\{e\})}$, else by $A \setminus \{e\}$.
\end{itemize}

\noindent
Local dynamics for the random-cluster model are currently very poorly understood (but see \cite{GeS} for the special case of graphs with bounded tree-width). However, in a recent surprising development, Ullrich~\cite{Ullrich1,Ullrich2} showed
that the mixing time of the heat-bath dynamics
on any graph differs from that of the SW dynamics by at most a $\poly(n)$ factor.
Thus the previously known bounds for SW translate to bounds for the heat-bath
dynamics for integer $q$.  By adapting Ullrich's technology to our CM setting, we are able to
obtain a similar translation of our results, thus establishing the first non-trivial bounds on
the mixing time of the mean-field heat-bath dynamics for all $q > 1$.
\begin{thm}\label{thm:intro3}
	For any $q > 1$, the mixing time of the heat-bath dynamics for the mean-field
	random-cluster model is $\Otilde(n^4)$ for $\lambda\notin [\ls,\lS]$, and
	$e^{\Omega(\sqrt{n})}$ for $\lambda\in (\ls,\lS)$.
\end{thm}
\noindent
The $\Otilde$ here hides polylogarithmic factors. We conjecture that the upper bound should be $\Otilde(n^2)$ for all
$\lambda\notin [\ls,\lS]$; the additional $n^2$ factor is inherent in Ullrich's spectral approach.




We conclude this introduction with some brief remarks about our techniques. Both
our upper and lower bounds on the mixing time of the CM dynamics focus on
the evolution of the one-dimensional random process given by the size of the largest component (which approaches $\theta n$
for $\lambda>\lambda_c$ and $\Theta(\log n)$ for $\lambda<\lambda_c$).
A key ingredient in our analysis is a function that describes the expected change,
or ``drift", of this random process at each step; the critical points $\ls$ and~$\lS$
discussed above arise naturally from consideration of the zeros of this drift function.

For our upper bounds, we construct a multiple-phase coupling of the evolution
of two arbitrary configurations, showing that they converge in $O(\log n)$
steps; this coupling is similar in flavor to that used by Long {\it et al.}~\cite{LNNP}
for the SW dynamics for $q=2$, but there are significant additional complexities in that
our analysis has to identify the ``slow mixing" window $(\ls,\lS)$ for $q>2$, and 
also has to contend with the fact that only a subset of the vertices 
(rather than the whole graph, as in SW) are active at each step. This latter issue is handled using precise concentration bounds for the number of active vertices, tailored estimates for the component structure of random graphs and a new coupling for pairs of binomial random variables.

For our exponential lower bounds we use the drift function to identify the metastable states mentioned ealier from which the dynamics cannot easily escape.
For both upper and lower bounds,
we have to handle the sub-critical and super-critical cases, $\lambda<\lambda_c$
and $\lambda>\lambda_c$, separately, even though our final results are insensitive to~$\lambda_c$, because the structure of typical configurations differs in the two cases. 

\section{Preliminaries}\label{sec:preliminaries}

In this section we gather a number of standard definitions and background results that we will refer to repeatedly in our proofs.

\subsection{Concentration bounds}

\begin{thm}[Chernoff Bounds] Let $X_1,...,X_k$ be independent Bernoulli random variables. Let $X = \sum_i X_i $ and $\mu = \E[X]$; then for any $\delta \in (0,1)$,
	\[\Pr[|X - \mu| > \delta \mu] \le 2\exp\left(-\frac{\delta^2\mu}{4}\right).\]
\end{thm}

\begin{thm}[Hoeffding's Inequality] Let $X_1,...,X_k$ be independent random variables such that $\Pr[X_i \in [a_i,b_i]] = 1$. Let $X = \sum_i X_i $ and $\mu = \E[X]$; then for any $\delta > 0$,
	\[\Pr[|X - \mu| > \delta] \le 2\exp\left(-\frac{2\delta^2}{\sum_{i=1}^k (b_i-a_i)^2}\right).\]
\end{thm}

\subsection{Mixing time}
Let $P$ be the transition matrix of a finite, ergodic Markov chain $M$ with state space $\Omega$ and stationary distribution $\pi$. The mixing time of $M$ is defined by
\[\taumix = \max\limits_{z \in \Omega}\min\limits_t \left\{ ||P^t(z,\cdot)-\pi(\cdot)||_{\textsc{\tiny TV}} \le 1/4 \right\}\]
where $||\mu-\nu||_{\textsc{\tiny TV}} = \max_{A \subset \Omega} |\mu(A)-\nu(A)|$ is the total variation distance between the distributions $\mu$ and $\nu$.

A {\it (one step) coupling} of the Markov chain $M$ specifies for every pair of states $(X_t, Y_t) \in \Omega^2$ a probability distribution over $(X_{t+1}, Y_{t+1})$ such that the processes $\{X_t\}$ and $\{Y_t\}$, viewed in isolation, are faithful copies of $M$, and if $X_t=Y_t$ then $X_{t+1}=Y_{t+1}.$ The {\it coupling time} 
is defined by
\[\Tcoup = \max_{x,y \in\Omega}\min\limits_t \{X_t=Y_t|X_0=x,Y_0=y\}.\]
For any $\delta \in (0,1),$ the following standard inequality (see, e.g., \cite{LPW}) provides a bound on the mixing time:
\begin{eqnarray}
\label{mt}
\taumix \le \min\limits_t \left\{\Pr[\Tcoup > t] \le 1/4\right\} \le O\left(\delta^{-1}\right) \cdot \min\limits_t \left\{\Pr[\Tcoup > t] \le 1-\delta\right\}. 
\end{eqnarray}

\subsection{Random graphs}\label{subsection:random-graphs}

Let $G_d$ be distributed as a $G(n,p=d/n)$ random graph where $d > 0$. We say that $d$ is bounded away from 1 if there exists a constant $\xi$ such that $|d-1|\ge\xi$. Let ${\cal L}(G_d)$ denote the largest component of $G_d$ and let $L_i(G_d)$ denote the {\it size} of the $i$-th largest component of $G_d$. (Thus, $L_1(G_d) = |{\cal L}(G_d)|$.) In our proofs we will use several facts about the random variables $L_i(G_d)$, which we gather here for convenience. We provide proofs for those results that are not available in the random graph literature.

\begin{lemma}[\cite{LNNP}, Lemma 5.7]\label{rg-isolated} Let $I(G_d)$ denote the number of isolated vertices in $G_d$. If $d = O(1)$, then there exists a constant $C > 0$ such that $\Pr[I(G_d) > Cn] = 1 - O\left(n^{-1}\right)$.
\end{lemma}

\begin{lemma}\label{rg-large} If $d = O(1)$, then $L_2(G_d) < 2n^{11/12}$ with probability $1 - O\left(n^{-1/12}\right)$ for sufficiently large $n$.
\end{lemma}

\begin{proof} If $d \le 1 + n^{-1/12}$, then by Theorem 5.9 in \cite{LNNP} (with $A^2 = c^{-1}\log n$ and $\epsilon = n^{-1/12}/2$), $L_1(G_d) < 2n^{11/12}$ with probability $1-O(n^{-1})$. When $d > 1 + n^{-1/12}$ we bound $L_2(G_d)$ using Theorem 5.12 in \cite{JLR}. Observe that this result applies to the random graph model $G(n,M)$ where an instance $G_M$ is chosen u.a.r. from the set of graphs with $n$ vertices and $M$ edges. The $G(n,p)$ and $G(n,M)$ models are known to be essentially equivalent when $M \approx \binom{n}{2}p$ and we can easily transfer this result to our setting. 
	
	Let $M_d$ be the number of edges in $G_d$ and $I = [\binom{n}{2}p - \sqrt{8 d n \log n},\binom{n}{2}p + \sqrt{8 d n \log n}]$; by Chernoff bounds, 
	\begin{eqnarray}\Pr[L_2(G_d) > n^{2/3}] 
	&\le& \sum_{m \in I} \Pr[L_2(G_m) > n^{2/3}] \Pr[M_d = m] + O(n^{-1}). \nonumber
	\end{eqnarray}
	Let $s = m - n/2$ as in \cite{JLR}; since $d > 1 + n^{-1/12}$, then $s \ge \frac{n^{11/12}}{4}$ for $m \in I$ and $n$ sufficiently large. Theorem 5.12 in \cite{JLR} implies that $\Pr[L_2(G_m) > n^{2/3}] = O(n^{-1/12})$; thus, $L_2(G_d) < 2n^{11/12}$ with probability $1-O(n^{-1/12})$. \end{proof}

\begin{lemma}[\cite{CDFR}, Lemma 7] \label{rg-sub-1} If $d < 1$ is bounded away from 1, then $L_1(G_d) = O(\log n)$ with probability $1-O\left(n^{-1}\right).$
\end{lemma}

\noindent
For $d > 1$, let $\gc = \gc(d)$ be the unique positive root of the equation 
\begin{equation}
e^{-d x} = 1 - x.\label{rg-root-eq}
\end{equation}
(Note that this equation has a positive root iff $d > 1$; see, e.g., \cite{JLR}.)
\begin{lemma}\label{lemma:rg-sup} Let $\tGln$ be distributed as a $G(n+m,d_n/n)$ random graph where $|m|=o(n)$ and $\lim\limits_{n \rightarrow \infty} d_n = d$. Assume $1 < d_n = O(1)$ and $d_n$ is bounded away from 1 for all $n \in \N$. Then,
	\begin{enumerate}[(i)]
		\item $L_2(\tGln) = O(\log n)$ with probability $1-O\left(n^{-1}\right).$ 
		\item For $A=o(\log n)$ and sufficiently large $n$, there exists a constant $c>0$ such that 
		\begin{equation}\label{sup-gc-eq}
		\Pr[|L_1(\tGln) - \gc(d) n| > |m| + A\sqrt{n}] \le e^{-c A^2}.
		\end{equation}
	\end{enumerate}
\end{lemma}
\begin{proof}
	Part $(i)$ follows immediately from Lemma 7 in \cite{CDFR}. For Part $(ii)$, let $M=n+m$ and $d_M=d_nM/n.$ By Lemma 11 in \cite{CDFR}, there exists a constant $c > 0$ such that
	\begin{eqnarray}
	e^{-c A^2} &\ge& \Pr[|L_1(\tGln) - \gc(d_M) M| > A\sqrt{M}] \nonumber\\
	&\ge& \Pr[|L_1(\tGln) - \gc(d_M) n| > |\gc(d_M) m| + A\sqrt{M}] \nonumber\\
	&\ge& \Pr[|L_1(\tGln) - \gc(d) n| > |\gc(d) n - \gc(d_M)n| + |\gc(d_M) m| + A\sqrt{M}]. \nonumber
	\end{eqnarray}
	Now, since $d_M \rightarrow d$, by continuity$\gc(d_M) \rightarrow \gc(d)$ as $n \rightarrow \infty$. Therefore, for a sufficiently large $n$,
	\[\Pr[|L_1(\tGln) - \gc(d) n| > |\gc(d_M) m| + 3 A\sqrt{n}] \le e^{-c A^2}\]
	and the result follows since $\gc(d_M) \le 1$. \end{proof}

\begin{cor} \label{cor:rg-sup} With the same notation as in Lemma \ref{lemma:rg-sup}, 
	\[|\E[L_1(\tGln)] - \gc(d) n| < |m| + O(\sqrt{n}).\]
\end{cor}

\begin{proof} Follows immediately by integrating (\ref{sup-gc-eq}).
\end{proof}

\begin{lemma} \label{sup-gc-ld} Consider a $\Gln$ random graph where $\lim_{n \rightarrow \infty} d_n = d$. Assume $1 < d_n = O(1)$ and $d_n$ is bounded away from 1 for all $n \in \N$. Then, for any constant $\varepsilon \in (0,1)$ there exists a constant $c(\varepsilon) > 0$ such that, for sufficiently large $n$, 
	\begin{equation}
	\Pr[|L_1(\Gln) - \gc(d_n) n| > \varepsilon n] \le e^{-c(\varepsilon) n}. \nonumber
	\end{equation}
\end{lemma}
\begin{proof} This result follows easily from Lemma 5.4 in \cite{LNNP}. Let $a_1$ and $a_2$ be constants such that $d \in (\gamma_1,\gamma_2)$. Since $\{d_n\} \rightarrow d,$ there exists $N \in \N$ such that $d_n \in (\gamma_1,\gamma_2)$ for all $n > N$. 
	
	By Lemma 5.4 in \cite{LNNP} (with $A = \varepsilon \sqrt{n}$), there exist constants $c_1(\varepsilon)$, $c_2(\varepsilon) > 0$ such that
	$\Pr[L_1(G_{\gamma_1}) < \gc(\gamma_1) n  - \varepsilon n] \le \exp(-c_1(\varepsilon) n)$ and
	$\Pr[L_1(G_{\gamma_2}) > \gc(\gamma_2) n  + \varepsilon n] \le \exp(-c_2(\varepsilon) n)$. By monotonicity $\gc(\gamma_2) > \gc(d_n) > \gc(\gamma_1)$,  and by continuity we can choose $\gamma_1$ and $\gamma_2$ sufficiently close to each other such that $|\gc(\gamma_2) - \gc(\gamma_1)| < \varepsilon.$ Observe also that $L_1(G_{\gamma_2}) \succeq L_1(\Gln) \succeq L_1(G_{\gamma_1})$, where $\succeq$ indicates stochastic domination\footnote{For distributions $\mu$ and $\nu$ over a partially ordered set $\Gamma$, we say that $\mu$ stochastically
		dominates $\nu$ if $\int g\;d\nu \le \int g\;d\mu$ for all increasing functions $g : \Gamma \rightarrow \R$.}. Thus, 
	\[e^{-c_1(\varepsilon) n} \ge \Pr[L_1(G_{\gamma_1}) < \gc(\gamma_1) n  - \varepsilon n] \ge \Pr[L_1(G_{d_n}) < \gc(\gamma_1) n  - \varepsilon n] \ge \Pr[L_1(G_{d_n}) < \gc(d_n) n  - 2 \varepsilon n]\] 
	and similarly,
	\[e^{-c_2(\varepsilon) n} \ge \Pr[L_1(G_{\gamma_2}) > \gc(\gamma_2) n  + \varepsilon n] \ge \Pr[L_1(G_{d_n}) > \gc(\gamma_2) n  + \varepsilon n] \ge \Pr[L_1(G_{d_n}) > \gc(d_n) n  + 2 \varepsilon n].\]
	Hence, there exist a constant $c(\varepsilon)$ such that $\Pr[|L_1(\Gln) - \gc(d_n) n| > \varepsilon n] \le e^{-c(\varepsilon) n}$. \end{proof}

\begin{lemma} \label{sub-strong} Assume $d$ is bounded away from 1. If $d < 1$, then $L_1(G_d) = O(\sqrt{n})$ with probability $1 - e^{-\Omega(\sqrt{n})}.$ If $d > 1,$ then $L_2(G_d) = O(\sqrt{n})$ with probability $1 - e^{-\Omega(\sqrt{n})}$.
\end{lemma}
\begin{proof} When $d < 1$ the result follows immediately from Lemma 6 in \cite{GJ}. When $d > 1$, by Lemma \ref{sup-gc-ld}, $L_1(G_d) \in I = [(\gc(d)-\varepsilon) n,(\gc(d)+\varepsilon) n]$ with probability $1-e^{-\Omega(n)}$. Conditioning on $L_1(G_d) = m$, by the {\it discrete duality principle} (see, e.g., \cite{H}) the remaining subgraph is distributed as a $G(n-m,d/n)$ random graph which is sub-critical for $m \in I$ and $\varepsilon$ sufficiently small. Therefore as for $d < 1$, $L_2(G_d) = O(\sqrt{n})$ with probability $1-e^{-\Omega(\sqrt{n})}$ as desired. \end{proof}

\begin{lemma} \label{rg-suscept} Assume $d$ is bounded away from 1. If $d < 1,$ then $\sum_{i \ge 1} L_i(G_d)^2 = O(n)$ with probability $1-O\left(n^{-1}\right).$ If $d > 1,$ then $\sum_{i \ge 2} L_i(G_d)^2 = O(n)$ with probability $1-O\left(n^{-1}\right).$
\end{lemma}
\begin{proof} When $d < 1$ the result follows by Chebyshev's inequality from Theorem 1.1 in \cite{JM}. When $d > 1$ the result follows from the discrete duality principle as in Lemma \ref{sub-strong}. \end{proof}

\subsection{The random-cluster model}

Recall from the introduction that the mean-field random-cluster model exhibits a phase transition
at $\lambda=\lambda_c(q)$ (see~\cite{BGJ}): in the sub-critical regime $\lambda < \lambda_c$
the largest component is of size $O(\log n)$, while in the super-critical regime $\lambda > \lambda_c$
there is a unique giant component of size {\raise.27ex\hbox{$\scriptstyle\mathtt{\sim}$}}$\theta_r n$, where $\gcrc = \gcrc(\lambda,q)$ is the largest $x > 0$ satisfying the equation 
\begin{eqnarray}
e^{-\lambda x} = 1 - \frac{q x}{1+(q-1)x}. \label{rc-gc-1}
\end{eqnarray}
(Note that, as expected, this equation is identical to (\ref{rg-root-eq}) when $q=1$, and $\gcrc(\lambda,q) < \gc(\lambda)$ for all $q > 1$.) The following is a more precise statement of this fact. 

\begin{lemma}[\cite{BGJ}]\label{rc-sub} \label{rc-sup}
	Let $G$ be distributed as a mean-field random-cluster configuration
	where $\lambda > 0$ and $q > 1$ are constants independent of $n$. 
	If $\lambda < \lambda_c$, then $L_1(G) = O(\log n)$ w.h.p. If $\lambda > \lambda_c$, then w.h.p.\ $|L_1(G)-\gcrc n| = O(n \omega^{-1}(n))$ for some sequence $\omega(n)$ satisfying $\omega(n) \rightarrow \infty$. 
\end{lemma}

\noindent
More accurate versions of this result can readily be obtained by combining the techniques from~\cite{BGJ} with stronger error bounds for random graph properties \cite{J}. We will use the following version in our proofs which we defer to Section \ref{subsection:drift}.

\begin{cor}\label{cor:rc-sub} If $\lambda > q$, then $|L_1(G)-\gcrc n| = O(n^{8/9})$ w.h.p.
\end{cor}

\subsection{Drift function} \label{subsection:drift}

As indicated in the introduction, our analysis relies heavily on understanding
the evolution of the size of the largest component under the CM
dynamics. To this end, for fixed $\lambda$ and $q$ let $\phi(\vf)$ be the largest $x > 0$ satisfying the equation 
\begin{eqnarray}
e^{-\lambda x} = 1 - \frac{q x}{1+(q-1)\vf}. \label{phi-definition}
\end{eqnarray}
Note this equation corresponds to (\ref{rg-root-eq}) for a $G\left(\left(\vf+\frac{1-\vf}{q}\right)n,\lambda/n\right)$ random graph, so
\begin{equation}
\phi(\vf) = \gc\left(\frac{\lambda(1+(q-1)\vf)}{q}\right).\label{phi:definition}
\end{equation}
Thus, $\phi$ is well-defined when $\lambda(1+(q-1)\vf) > q$. In particular, $\phi$ is well-defined in the interval $(\theta_{{\rm min}},1]$,
where $\theta_{\rm min} = \max\left\{(q-\lambda)/\lambda(q-1),0\right\}$. 

We will see in Sections \ref{upper-sub} and \ref{section-sup} that for a configuration with a unique ``large" component of size $\theta n$, the expected ``drift" in the size of the largest component will be determined by the sign of the function $f(\vf) = \vf - \phi(\vf)$: $f(\vf) > 0$ corresponds to a negative drift and $f(\vf) < 0$ to a positive drift. Thus, let 
\[\lambda_s = \max\{\lambda \le \lambda_c: f(\vf) > 0~~\forall\vf\in(\theta_{\rm min},1]\}~{\rm and,}\]
\[\lambda_S = \min \{\lambda \ge \lambda_c: f(\vf)(\vf-\gcrc) > 0~~\forall \vf \in (\theta_{\rm min},1]\}.\]
Intuitively, $\lambda_s$ and $\lambda_S$ are the maximum and minimum values, respectively, of $\lambda$ for which the drift in the size of the largest component is always in the required direction (i.e., towards 0 in the sub-critical case and towards $\gcrc n$ in the super-critical case).  

The following lemma, which we will prove shortly, reveals basic information about the quantities $\lambda_s$ and $\lambda_S$.
\begin{lemma} \label{ls} For $q \le 2,$ $\lambda_s = \lambda_c = \lambda_S = q$; and for $q > 2$, 
	$\lambda_s < \lambda_c < \lambda_S = q$.
\end{lemma}

\noindent
For integer $q\ge 3$, $\lambda_s$ corresponds to the threshold~$\beta_s$ in the mean-field $q$-state Potts model at which the local
(Glauber) dynamics undergoes an exponential slowdown \cite{CDLLPS}. In fact, a change of variables reveals that $\lambda_s = 2\beta_s$ for the specific mean-field Potts model normalization in \cite{CDLLPS}.

In Figure \ref{fig:f} we sketch $f$ in its only two qualitatively different regimes: $q \le 2$ and $q > 2$. The following lemma provides bounds for the drift of the size of the largest component under CM steps.

\begin{figure}[t!]
	\begin{center}
		\includegraphics[page=1,clip,trim=85 562 80 70]{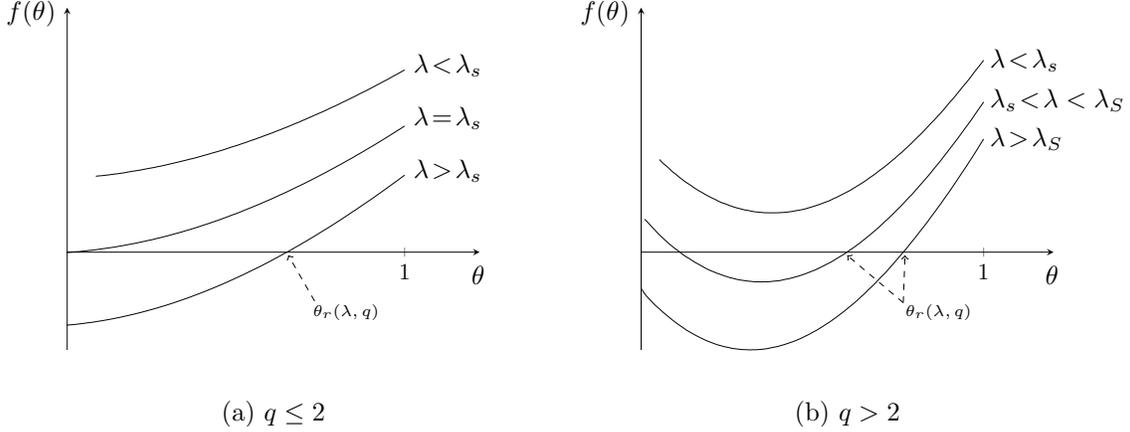}
		\caption{Sketch of the function $f$.}
		\label{fig:f}
	\end{center}
\end{figure}


\begin{lemma} \label{drift-bounds} For all $\vf \in (\theta_{\rm min},1]$,
	\begin{enumerate}[(i)] 
		\item If $\lambda < \lambda_s$, there exists a constant $\delta > 0$ such that $f(\vf) \ge \delta$.
		\item When $\lambda > \lambda_S$, if $\vf > \gcrc$, then $\vf \ge \phi(\vf) \ge \gcrc$ and if $\vf < \gcrc,$ then $\vf \le \phi(\vf) \le \gcrc.$ 
		\item If $\lambda > \lambda_S$, there exists a constant $\delta \in (0,1)$ such that $\delta |\vf-\gcrc| \le |\phi(\vf)-\vf|$.
	\end{enumerate} 
\end{lemma}

\noindent
Before proving Lemmas \ref{ls} and \ref{drift-bounds} we establish the following useful facts about the functions $\phi$ and $f$ which in most cases follow easily from their definitions. 

\begin{fact} \label{drift-phi-prop}\
	\begin{enumerate}[(i)]
		\item $\vf^* \in (\theta_{\rm min},1]$ is a fixed point of $\phi$ if and only if $\vf^*$ is a solution of (\ref{rc-gc-1}).
		\item $\phi$ is continuous, differentiable, strictly increasing and strictly concave in $(\theta_{\rm min},1]$.
		\item $\phi'(\vf) > \frac{q-1}{q}$ for all $\vf \in (\theta_{\rm min},1]$.
	\end{enumerate} 
\end{fact}
\begin{proof} Obviously any fixed point of $\phi$ is also a solution of (\ref{rc-gc-1}). For the other direction, consider the injective function $h(x) = \frac{x}{1-e^{-\lambda x}};$ if $\vf^*$ is a root of equation (\ref{rc-gc-1}), then $h(\vf^*) = h(\phi(\vf^*))$ and $\phi(\vf^*)=\vf^*$. 
	
	
	By differentiating both sides of (\ref{phi-definition}),
	\[\phi'(\vf) = \frac{q-1}{q}\cdot\frac{(1-e^{-\lambda\phi(\vf)})^2}{1-e^{-\lambda\phi(\vf)}-\lambda\phi(\vf)e^{-\lambda\phi(\vf)}}\]
	which implies that $\phi$ is differentiable and continuous. Since $e^{-\lambda \phi(\vf)} > 1 - \lambda \phi(\vf)$, then $\phi'(\vf) > \frac{q-1}{q}$ and $\phi$ is strictly increasing. Finally, consider the function
	\begin{eqnarray}
	g(x) = \frac{qx}{(q-1)(1-e^{-\lambda x })} - \frac{1}{q-1} - x \label{drift-function-h}. \nonumber
	\end{eqnarray}
	By solving for $\vf$ in (\ref{phi-definition}), observe that $g(\phi(\vf)) = \vf - \phi(\vf)$ for $\vf \in (\theta_{\rm min},1]$. Therefore,
	\[\phi''(\vf) = -g''(\phi(\vf))\phi'(\vf)(1+g'(\phi(\vf)))^{-2}\]
	and a straightforward calculation shows that $g'' > 0$ in $(0
	,1]$. Consequently, $\phi$ is strictly concave in $(\theta_{\rm min},1]$.
\end{proof}

\begin{fact}\label{drift-f-facts} \
	\begin{enumerate}[(i)]
		\item $f$ is continuous, differentiable and strictly convex in $(\theta_{\rm min},1]$.
		\item $f(\gcrc) = 0$, $f(1)>0$ and $f'(\vf) < 1/q$ for all $\vf \in (\theta_{\rm min},1]$.
		\item Let $f(\theta_{\rm min}^+) = \lim_{\vf \rightarrow \theta_{\rm min}} f(\vf);$ then $\sgn(f(\theta_{\rm min}^+)) = \sgn(q-\lambda)$.
	\end{enumerate}
\end{fact}
\begin{proof} Parts $(i)$ and $(ii)$ follow immediately from Fact \ref{drift-phi-prop}. For Part $(iii)$, observe that when $\lambda > q$, $\theta_{\rm min} = 0$ and the function $\phi$ is defined at 0; thus, $f(\theta_{\rm min}^+) = -\phi(0) < 0$. When $\lambda < q$, $\theta_{\rm min} = (q-\lambda)/\lambda(q-1)$ and by continuity, $\lim_{\vf \rightarrow \theta_{\rm min}} \phi(\vf) = 0$; hence, $f(\theta_{\rm min}^+) = \theta_{\rm min} > 0$.
\end{proof}

\noindent
Observe that if $\vf^*$ is a zero of $f$, then $\vf^*$ is a fixed point of $\phi$ and consequently a root of equation (\ref{rc-gc-1}).
Lemma 2.5 from \cite{BGJ} dissects the roots of equation (\ref{rc-gc-1}) and hence identifies the roots of $f$ in $(\theta_{\rm min},1]$.

\begin{fact}\label{f-roots}\ The roots of the function $f$ in $(\theta_{\rm min},1]$ are given as follows: 
	\begin{enumerate}[(i)]
		\item When $q \le 2$: if $\lambda \le \lambda_c$, $f$ has no positive roots and if $\lambda > \lambda_c$, $f$ has a unique positive root.
		\item When $q > 2$, there exists $\lambda_{\rm min} < \lambda_c$ such that: if $\lambda < \lambda_{\rm min}$, $f$ has no positive roots; if $\lambda_{\rm min} < \lambda < q$, $f$ has exactly two positive roots; and if $\lambda > q$, $f$ has a unique positive root.
	\end{enumerate}
\end{fact}

\begin{proof}[{\bf Proof of Lemma \ref{ls}:}]
	Since $f(1) > 0$, by continuity $f$ is strictly positive in $(\theta_{\rm min},1]$ if and only if $f$ has no roots in $(\theta_{\rm min},1].$ When $q \le 2$, by Fact \ref{f-roots}, if $\lambda \le \lambda_c$ then $f$ has no roots in $(\theta_{\rm min},1]$, and if $\lambda > \lambda_c$ then $f$ has a unique root in $(\theta_{\rm min},1]$; thus, $\ls = \lambda_c = q$. When $q > 2,$ by Fact \ref{f-roots}, $\ls = \lambda_{\rm min} < \lambda_c.$
	
	If $\lambda > q$, then $f(\theta_{\rm min}^+) < 0$ and Fact \ref{f-roots} implies that $f$ has a unique root in $(\theta_{\rm min},1]$. Hence $f$ is negative in $(\theta_{\rm min},\gcrc)$ and positive in $(\gcrc,1]$ and then $\lambda_S \le q$. For $q \le 2$ this readily implies $\ls = \lambda_c = \lS = q$. For $q > 2$, if $q > \lambda > \lambda_c$, Fact \ref{f-roots} implies that $f$ has exactly two positive roots in $(\theta_{\rm min},1]$. Recall that $f(\gcrc) = 0$ and let $\vf^*$ be the other root of $f$ in $(\theta_{\rm min},1]$; by the definition of $\gcrc$, $\vf^* < \gcrc$. Moreover, $f(1) > 0$ and $f(\theta_{\rm min}^+) > 0$ since $q > \lambda$. Therefore, $f$ is positive in $(0,\vf^*) \cup (\gcrc,1]$ and negative in $(\vf^*,\gcrc)$. If $\theta < \vf^*$, then $f(\vf)(\vf-\gcrc) < 0$; thus, $\lambda_S = q$. 
\end{proof}

\begin{proof}[{\bf Proof of Lemma \ref{drift-bounds}:}] If $\lambda < \lambda_s$, then $f(\vf) > 0$ for all $\vf \in (\theta_{\rm min},1]$ by definition. Also, $f$ is continuous in $(\theta_{\rm min},1]$ and $f(\theta_{\rm min}^+) > 0;$ thus, $f$ must attain a minimum value $\delta > 0$ in $(\theta_{\rm min},1]$ which implies Part $(i)$. Part $(ii)$ follows from the definition of $\lS$ and the fact that $\phi$ is increasing in $(\theta_{\rm min},1]$.
	
	The function $f$ is  continuous, differentiable and convex in $(\theta_{\rm min},1]$, so it lies above all of its tangents. Observe that $f(\theta_{\rm min}^+) < 0$ when $\lambda > \lambda_S = q.$ Let $T$ be the line tangent to $f$ at $\gcrc$. Observe that $f'(\gcrc) > 0$ since $f$ is convex in $(\theta_{\rm min},1]$ and $f(\theta_{\rm min}^+) < 0$. Let $M = \min \{f'(\gcrc),-f(\theta_{\rm min}^+)/\gcrc\};$ by Fact \ref{drift-f-facts}, $f' < 1/q$ and so $M \in (0,1/q]$. Consider the line $S(\vf) = \frac{M}{2}(\vf-\gcrc)$ and the line $R$ going through the points $(0,f(\theta_{\rm min}^+))$ and $(\gcrc,0).$ The slope of $R$ is $-f(\theta_{\rm min}^+)/\gcrc,$ and the lines $S$, $R$ and $T$ intersect at $(\gcrc,0)$. Therefore, $S$ lies above $R$ in $(0,\gcrc)$ and below $T$ in $(\gcrc,1].$ By convexity, $f$ lies below $R$ in $(0,\gcrc)$ and above $T$ in $(\gcrc,1].$ Thus, $S$ lies above $f$ in $(0,\gcrc)$ and below $f$ in $(\gcrc,1].$ Therefore, if $\vf < \gcrc$ then $\frac{M}{2}(\vf-\gcrc) > \vf - \phi(\vf)$ and if $\vf > \gcrc$ then $\frac{M}{2}(\vf-\gcrc) < \vf - \phi(\vf)$. Part $(iii)$ then follows by taking $\delta = M/2$.
\end{proof}

\noindent
The following fact will also be helpful.

\begin{fact} \label{drift-bounds-1} If $\lambda > q$, then $\gcrc > 1 - q/\lambda.$ 
\end{fact}
\begin{proof} By solving for $\lambda$ in (\ref{rc-gc-1}), it is sufficient to show that
	\[q > \frac{1-x}{x} \ln \left(\frac{1+(q-1)x}{1-x}\right) = h(x)\]
	for $x \in [0,1]$. A straightforward calculation shows that $h$ is decreasing in $(0,+\infty)$ and that $\lim\limits_{x \rightarrow 0} h(x) = q$. \end{proof}

\noindent
Finally, we can use the results in this subsection to prove Corollary \ref{rc-sup} stated in the previous subsection.

\begin{proof}[{\bf Proof of Corollary \ref{cor:rc-sub}:}] 
	
	By Lemma 3.2 in \cite{BGJ}, $L_2(G) < n^{3/4}$ w.h.p. Conditioning on this event, independently color each component of $G$ red with probability $1/q$. Let $L_r$ denote the size of the largest red component and $n_r$ the total number of red vertices.
	
	Let $\Gamma_{\rm \theta}$ be the intersection of the events that ${\cal L}(G)$ is colored red and $L_1(G) = \vf n$ where $\vf n \in \N$. Observe that $\Pr[L_r = \vf n \,|\, \Gamma_{\rm \theta}] = 1$, and by Hoeffding's inequality $\Pr[n_r\in J \,|\, \Gamma_{\rm \theta}] = 1 - O(n^{-2})$ where $J := \left[\left(\vf + \frac{1-\vf}{q}\right)n-\xi,\left(\vf + \frac{1-\vf}{q}\right)n+\xi\right]$ with $\xi = \sqrt{n^{7/4}\log n}$. Putting these two facts together,
	\begin{eqnarray}
	\frac{1}{2q}\Pr[L_1(G) = \vf n] \le \Pr[n_r \in J \,|\, \Gamma_{\rm \theta} ] \Pr[\Gamma_{\rm \theta}] \le \Pr[L_r = \vf n,n_r \in J]. \nonumber
	\end{eqnarray}
	By Lemma 3.1 in \cite{BGJ}, conditioned on the red vertex set, the red subgraph is distributed as a $G(n_r,p)$ random graph, so
	\begin{eqnarray}
	\frac{1}{2q}\Pr[L_1(G) = \vf n] \le \sum_{m \in J} \Pr[L_r = \vf n|n_r = m]\Pr[n_r=m]  \le  \max_{m \in J}\Pr[\ell(m) = \vf n] \nonumber
	\end{eqnarray}
	where $\ell(m)$ is distributed as the size of the largest component of a $G(m,p)$ random graph. Note that for $m \in J$ the random graph $G(m,p)$ is super-critical because $\lambda > q$. Since $\xi = \sqrt{n^{7/4}\log n}$, by (\ref{phi:definition}) and Lemma \ref{lemma:rg-sup} with $A=\sqrt{n^{3/4}\log n}$, $\Pr[|\ell(m) - \phi(\vf) n| > 2\xi] = O(n^{-2})$. Since $\lambda > q = \lS$, Lemma \ref{drift-bounds} implies that there exists a constant $\delta \in (0,1)$ such that $|\vf-\phi(\vf)| > \delta|\vf-\gcrc|$. Thus, if $|\vf - \gcrc| n > n^{8/9}$, then $\Pr[L_1(G) = \vf n] = O(n^{-2})$. The result follows by a union bound over all the positive integer values of $\vf n$ such that $|\vf - \gcrc| n > n^{8/9}$ and $\vf n \le n$. \end{proof}


\subsection{Binomial coupling}\label{bc}

In our coupling constructions we will use the following fact about the coupling of two binomial random variables.

\begin{lemma} \label{bc-str} 
	Let $X$ and $Y$ be binomial random variables with parameters $m$ and $r$, where $r \in (0,1)$ is a constant. Then, for any integer $y > 0,$ there exists a coupling $(X,Y)$ such that for a suitable constant $\gamma = \gamma(r)> 0$, 
	\[\Pr[X-Y=y] \ge 1 - \frac{\gamma y}{\sqrt{m}}.\]
	Moreover if $y = a\sqrt{m}$ for a fixed constant $a$, then $\gamma a < 1.$
\end{lemma}

\begin{proof} 
	This lemma is a slight generalization of Lemma 6.7 in \cite{LNNP} and, like that lemma, follows from a standard fact about symmetric random walks. When $y = \Theta(\sqrt{m})$ the result follows directly from Lemma 6.7 in \cite{LNNP}, so we assume $y < \sqrt{m}$ which will simplify our calculations.
	
	Let $X_1 ,..., X_m,Y_1,...,Y_m$ be Bernoulli i.i.d's with parameter $r$. Let $X = \sum_{i=1}^m X_i$, $Y = \sum_{i=1}^m Y_i$, and $D_k = \sum_{i=1}^k (X_i - Y_i)$. We construct a coupling for $(X,Y)$ by coupling each $(X_k,Y_k)$ as follows:
	\begin{enumerate}
		\item If $D_k \neq y$, sample $X_{k+1}$ and $Y_{k+1}$ independently.
		\item If $D_k = y$, set $X_{k+1} = Y_{k+1}$.
	\end{enumerate}
	Clearly this is a valid coupling since $X$ and $Y$ are both binomially distributed.
	
	If $D_k = y$ for any $k \le m$, then $X - Y = y$. Therefore,
	$\Pr[X-Y=y] \ge \Pr[M_m \ge y]$ where $M_m = \max\{D_0,...,D_m\}$. Observe that while $D_k \neq y$, $\{D_k\}$ behaves like a (lazy) symmetric random walk. The result then follows from the following fact:
	
	\begin{fact}\label{fact:srw} Let $\xi_1,...,\xi_m$ be i.i.d such that $\Pr[\xi_i=1]=\Pr[\xi_i=-1] = w$ and $\Pr[\xi_i=0] = 1-2w$. Let $S_k = \sum_{i=1}^k \xi_i$ and $M_k = \max\{S_1,...,S_k\}$. Then, for any positive integer $y < \sqrt{m}$, there exists a constant $\gamma = \gamma(r)> 0$ such that
		\[\Pr[M_m \ge y] \ge 1 - \frac{\gamma y}{\sqrt{m}}.\]
	\end{fact}
	
	\begin{proof} This is a well-known fact about symmetric random walks, so we just sketch one way of proving it. By the {\it reflection principle}, $\Pr[M_m \ge y] \ge 2\Pr[S_m > y]$ (see, e.g., \cite{GS}) and by the Berry-Ess\'{e}en inequality, $|\Pr[S_m >  k \sqrt{2w m}]-\Pr[N>k]| = O(m^{-1/2})$ where $N$ is a standard normal random variable (see, e.g., \cite{F}). The result follows from the fact that $2 \Pr[N > k] \ge 1 - \sqrt{\frac{2}{\pi}}k$. \end{proof}
	
	\noindent
	Note that in our case $w=r(1-r)$.\end{proof}

\subsection{Hitting time estimate for supermartingales}\label{subsection:hitting-times-sub}

We will require the following easily derived hitting time estimate for supermartingales.

\begin{lemma}
	\label{lemma:hitting-times-sub}
	Consider the stochastic process $\{Z_t\}$ such that $Z_t \in [-n,n]$ for all $t \ge 0$. Assume $Z_0 > a$ for some $a \in [-n,n]$ and let $T = \min\{t > 0 : Z_t \le a\}$. Suppose $\E[Z_{t+1}-Z_{t}|{\cal F}_t] \le -A$, where $A > 0$ and ${\cal F}_t$ is the history of the first $t$ steps. Then, $\E[T] \le 4n/A$.
\end{lemma}

\begin{proof}
	Let $Y_t = Z_t^2 - 4nZ_t - 2nAt$. A standard calculation reveals that $Y_t$ is a submartingale; i.e., $\E[Y_t-Y_{t-1}|{\cal F}_{t-1}] \ge 0$ for all $t > 0$. Observe also that $T$ is a stopping time, since the event $\{T=t\}$ depends only on the history up to time $t$.
	Since $Z_t \in [-n,n]$, $Y_t$ is bounded and thus the optional stopping theorem (see, e.g., \cite{D}) implies
	\[5n^2 - 2nA\E[T] \ge \E[Y_T] \ge \E[Y_0] \ge -3n^2.\]
	Hence,  $\E[T] \le 4n/A$, as desired.
\end{proof}

\section{Mixing time upper bounds}\label{sec:upper-bounds}

In this section we prove the upper bound portion of Theorem \ref{thm:intro1} from the introduction.

\begin{thm} Consider the CM dynamics for the mean-field random-cluster model with parameters $p=\lambda/n$ and $q$ where $\lambda > 0$ and $q > 1$ are constants independent of $n$. If $\lambda \not\in [\ls,\lS]$, then $\taumix = O(\log n)$.
\end{thm}

\begin{proof}[Proof Sketch] 
	
	Consider two copies $\{X_t\}$ and $\{Y_t\}$ of the CM dynamics starting from two arbitrary configurations $X_0$ and $Y_0$. We design a coupling $(X_t,Y_t)$ of the CM steps and show that $\Pr[X_T=Y_T] = \Omega(1)$ for some $T = O(\log n)$; the result then follows from (\ref{mt}). The coupling consists of four phases. In the first phase $\{X_t\}$ and $\{Y_t\}$ are run independently. In Section \ref{sub-sec:first-phase} we establish that after $O(\log n)$ steps $\{X_t\}$ and $\{Y_t\}$ each have at most one large component with probability $\Omega(1)$. We call a component {\it large} if it contains at least $2 n^{11/12}$ vertices; otherwise it is {\it small}. 
	
	In the second phase, $\{X_t\}$ and $\{Y_t\}$ also evolve independently. In Sections \ref{upper-sub} and \ref{section-sup} we show that, conditioned on the success of Phase 1, after $O(\log n)$ steps with probability $\Omega(1)$ the largest components in $\{X_t\}$ and $\{Y_t\}$ have sizes close to their expected value: $O(\log n)$ in the sub-critical case and 
	{\raise.27ex\hbox{$\scriptstyle\mathtt{\sim}$}}$\theta_r n$ in the super-critical case. In the third phase, $\{X_t\}$ and $\{Y_t\}$ are coupled to obtain two configurations with the same component structure. This coupling, described in Section \ref{upper-sc}, makes crucial use of the binomial coupling of Section \ref{bc}, and conditioned on a successful conclusion of Phase 2 succeeds with probability $\Omega(1)$ after $O(\log n)$ steps. In the last phase, a straightforward coupling is used to obtain two identical configurations from configurations with the same component structure. This coupling is described in Section \ref{sub-section:phase-four} and succeeds w.h.p.\ after $O(\log n)$ steps, conditioned on the success of the previous phases.   
	
	Putting all this together, there exists a coupling $(X_t,Y_t)$ such that, after $T=O(\log n)$ steps, $X_T = Y_T$ with probability $\Omega(1)$. The reminder of this section fleshes out the above proof sketch.
\end{proof}

\noindent
We now introduce some notation that will be used throughout the rest of the paper. As before, we will use ${\cal L}(X_t)$ for the largest component in $X_t$ and $L_i(X_t)$ for the {\it size} of the $i$-th largest component of $X_t$. (Thus, $L_1(X_t) = |{\cal L}(X_t)|$.) For convenience, we will sometimes write $\theta_t n$ for $L_1(X_t)$. Also, we will use ${\cal E}_t$ for the event that ${\cal L}(X_t)$ is activated, and $A_t$ for the number of activated vertices at time $t$.

\subsection{Convergence to configurations with a unique large component} \label{sub-sec:first-phase}

\begin{lemma}\label{unique-gc} For any starting random-cluster configuration $X_0$, there exists $T = O(\log n)$ such that $X_T$ has at most one large component with probability $\Omega(1)$. 
\end{lemma}
\begin{proof} 
	
Let $N_t$ be the number of new large components created in sub-step {\rm (iii)} of the CM dynamics at time $t$. If $A_t <  2 n^{11/12}$, then $N_t=0$. Together with Lemma \ref{rg-large} this implies that $\Pr[N_t>1|X_t,A_t=a] \le a^{-1/12}$ for all $a \in [0,n]$. Thus,
\begin{eqnarray}
\E[N_t|X_t] &=& \sum\limits_{a=0}^n \E[N_t|X_t,A_t=a]\Pr[A_t=a|X_t] \nonumber\\
&\le& \sum\limits_{a=0}^n \left(\Pr[N_t \le 1|X_t,A_t=a]+\frac{a}{2 n^{11/12}}\Pr[N_t>1|X_t,A_t=a]\right)\Pr[A_t=a|X_t] \nonumber\\
&\le& \sum\limits_{a=0}^n \left(1+\frac{a}{2 n^{11/12}}\frac{1}{a^{1/12}}\right)\Pr[A_t=a|X_t] \le 2. \nonumber
\end{eqnarray}
Let $K_t$ be the number of large components in $X_t$ and let $C_t$ be the number of activated large components in sub-step {\rm (i)} of the CM dynamics at time $t$. Then,
\begin{eqnarray}
\E[K_{t+1}|X_t] = K_t - \E[C_t|X_t] + \E[N_t|X_t] \le K_t - \frac{K_t}{q} + 2 \le \left(1-\frac{1}{2q}\right)K_t \nonumber
\end{eqnarray}
provided $K_t \ge 4q$. Assuming that $K_{t} \ge 4q$ for all $t < T$, we have
\[\E[K_t|X_0] \le \left(1-\frac{1}{2q}\right)^tK_0.\]
Hence, Markov's inequality implies that $K_T < 4q$ w.h.p.\ for some $T=O(\log n)$. If at time $T$ the remaining $K_{T}$ large components become active, then $K_{T+1} \le 1$ w.h.p.\ by Lemma \ref{rg-large}. All $K_{T}$ components become active simultaneously with probability at least $q^{-4q}$ and thus $K_{T+1} \le 1$ with probability $\Omega(1)$, as desired.
\end{proof}

\subsection{Convergence to typical configurations: the sub-critical case}\label{upper-sub}

\begin{lemma}\label{sub-contraction} Let $\lambda < \lambda_s$; if $X_0$ has at most one large component, then there exists $T=O(\log n)$ such that $L_1(X_T) = O(\log n)$ with probability $\Omega(1)$.
\end{lemma}
\noindent
Let $\xi = \sqrt{2 n^{23/12}\log n}$. The following fact will be used in the proof. 

\begin{fact} \label{no-large-comp} If $X_t$ has at most one large component, then for sufficiently large $n$ each of the following holds with probability $1-O(n^{-1})$: 
	\begin{enumerate}[(i)]
		\item If ${\cal L}(X_t)$ is inactive, then all new components in $X_{t+1}$ have size $O(\log n)$.
		\item If ${\cal L}(X_t)$ is active, then $A_t \in J_t:= \left[L_1(X_t)+\frac{n-L_1(X_t)}{q}-\xi,L_1(X_t)+\frac{n-L_1(X_t)}{q}+\xi\right]$.
		\item If there is no large component in $X_t$, then $L_1(X_{t+1}) = O(\log n)$.
	\end{enumerate}
\end{fact}

\begin{proof} Observe that $\E[A_t|X_t,\neg{\cal E}_t] = \frac{n-L_1(X_t)}{q} =: \mu$, and $\sum_{j \ge 2 } L_j(X_t)^2 < 2 n^{23/12}$ since $L_2(X_t) < 2 n ^{11/12}$. By Hoeffding's inequality:
\begin{eqnarray}
\Pr\left[\,|A_t - \mu| > \xi~|~X_t,\neg{\cal E}_t\right] \le 2\exp\left(-\frac{4 n^{23/12}\log n}{2 n^{23/12}}\right) \le \frac{2}{n^2} \label{active-hoef}. \nonumber
\end{eqnarray}
Thus, $A_t \le \frac{n-L_1(X_t)}{q}+\xi$ with probability at least $1 - O(n^{-2})$. Observe that \[\left(\frac{n-L_1(X_t)}{q}+\xi\right)\frac{\lambda}{n} < \frac{\lambda}{q} + o(1) < 1\]
for sufficiently large $n$; hence, the random graph $G(A_t,p)$ is sub-critical and part (i) follows from Lemma \ref{rg-sub-1}. Parts (ii) and (iii) follow in similar fashion.
\end{proof}

\begin{proof}[Proof of Lemma \ref{sub-contraction}] 
	
	
	Suppose $X_t$ has a unique large component. If ${\cal L}(X_t)$ is activated in sub-step {\rm (i)} of the CM dynamics, then Lemma \ref{rg-large} implies that $X_{t+1}$ has at most one large component with probability $1-O(n^{-1/12})$. Otherwise, if ${\cal L}(X_t)$ is not activated, $X_{t+1}$ will have a unique large component with probability $1-O(n^{-1})$ by Fact \ref{no-large-comp}(i). A union bound then implies that during $T$ consecutive steps of this phase configurations will have at most one large component w.h.p.\ for any $T = O(\log n)$. Thus, we condition on this event. 	
	
	For ease of notation set $\ts := \theta_{\rm min}$, with $\theta_{\rm min}$ defined as in Section \ref{subsection:drift}. Note that $(\ts + (1-\ts)q^{-1})\lambda = 1$. Hence, if $L_1(X_t) = \ts n$ and ${\cal L}(X_t)$ is activated, then the percolation step (sub-step {\rm (iii)} of the CM dynamics) is critical with non-negligible probability. This makes the analysis in the neighborhood of $\ts n$ more delicate. 
	
	We consider first the case where $\theta_t \ge \ts + \varepsilon$ for some small constant $\varepsilon > 0$ to be chosen later. By Fact \ref{no-large-comp}(i), if ${\cal L}(X_t)$ is inactive all the new components have size $O(\log n)$ with probability $1-O(n^{-1})$. Thus,
	\begin{equation}
	\E[L_1(X_{t+1}) \,|\, X_t,\neg{\cal E}_t] \le L_1(X_t) + O(1) = \theta_t n + O(1). \label{case2-con1}
	\end{equation}
	
	To bound $\E[L_1(X_{t+1}) \,|\, X_t,{\cal E}_t]$, let $h^+(\theta_t) = \theta_t n + (1-\theta_t)q^{-1}n + \xi$ and let $\ell^+(\theta_t)$ be a random variable distributed as the size of the largest component of a $G(h^+(\theta_t),p)$ random graph. Then, by Fact \ref{no-large-comp}(ii) we have 
	\begin{eqnarray}
	\E[L_1(X_{t+1})\,|\,X_t,{\cal E}_t] &\le& \sum_{a \in J_t} \E[L_1(X_{t+1})\,|\,X_t,{\cal E}_t,A_t=a] \Pr[A_t = a\,|\,X_t,{\cal E}_t] + O(1) \nonumber\\
	&\le& \E[L_1(X_{t+1})\,|\,X_t,{\cal E}_t,A_t=h^+(\theta_t)] + O(1) = \E[\ell^+(\theta_t)]+O(1). \nonumber\label{case2-gc}
	\end{eqnarray}
	When $\theta_t \ge \ts+\varepsilon$, $G(h^+(\theta_t),p)$ is a super-critical random graph:
	\begin{equation}
	h^+(\theta_t)\frac{\lambda}{n} \ge \left(\ts+\varepsilon + \frac{1-\ts-\varepsilon}{q}\right)\lambda = 1 + \left(1 - \frac{1}{q}\right)\varepsilon\lambda > 1.\label{upper:sup-bound}
	\end{equation}
	Thus, Corollary \ref{cor:rg-sup} implies
	\begin{equation}
	\E[L_1(X_{t+1})~|~X_t,{\cal E}_t] \le \phi(\theta_t) n + O(\xi), \label{case2-con2}
	\end{equation}
	where $\phi(\theta_t)$ is defined as in (\ref{phi-definition}). Since $\lambda < \ls$, by Lemma \ref{drift-bounds} there exists a constant $\delta > 0$ such that $\theta_t - \phi(\theta_t) \ge \delta$. Therefore, putting (\ref{case2-con1}) and (\ref{case2-con2}) together, we have
	\begin{equation}
	\E[L_1(X_{t+1})~|~X_t] \le \left(1 - \frac{1}{q}\right)\theta_tn + \frac{\phi(\theta_t)n}{q} + O(\xi) \le \theta_t n-\frac{\delta n}{q} + O(\xi). \label{gc-change-1}
	\end{equation}
	
	As mentioned before, in a close neighborhood of $\ts$ the percolation step is critical with non-negligible probability, so when $\theta_t \in (\ts-\varepsilon,\ts+\varepsilon)$ we use monotonicity to simplify the analysis. In particular, observe that $\E[\ell^+(\theta_t)] \le \E[\ell^+(\ts+\varepsilon)]$. By (\ref{upper:sup-bound}), the random graph $G(h^+(\ts+\varepsilon),p)$ is super-critical. Hence, Corollary \ref{cor:rg-sup} implies 
	$\E[\ell^+(\ts+\varepsilon)] \le \phi(\ts+\varepsilon)n + O(\xi)$.
	
	The bounds in (\ref{case2-con1}) and (\ref{case2-gc}) still hold for $\theta_t \in (\ts-\varepsilon,\ts+\varepsilon)$, so
	\begin{eqnarray}
	\E[L_1(X_{t+1})~|~X_t] &\le& \left(1-\frac{1}{q}\right)\theta_t n + \frac{\phi(\ts+\varepsilon)n}{q} + O(\xi) \nonumber\\
	&\le& \theta_tn + \frac{(\phi(\ts+\varepsilon)-(\ts+\varepsilon))n}{q} +  \frac{2\varepsilon n}{q} + O(\xi). \nonumber
	\end{eqnarray}
	By Lemma \ref{drift-bounds}, $(\ts + \varepsilon) - \phi(\ts+\varepsilon) \ge \delta$, and by choosing $\varepsilon$ sufficiently small we obtain (\ref{gc-change-1}) for $\theta_t \in (\ts-\varepsilon,\ts+\varepsilon)$. Thus, there exists a constant $\gamma >0$ such that for all $\theta_t > \ts-\varepsilon$: \[\E[L_1(X_{t+1})-L_1(X_{t})\,|\,X_t] \le -\gamma n.\]
	
	Let
	$\tau = \min\{t > 0: L_1(X_t) \le (\ts-\varepsilon) n\}$.
	By Lemma \ref{lemma:hitting-times-sub}, $\E[\tau] \le 4/\gamma$ and thus $\Pr[\tau > 8/\gamma] \le 1/2$ by Markov's inequality. Hence, $L_1(X_T) \le (\ts-\varepsilon)n$ for some $T=O(1)$ with probability $\Omega(1)$.
	
	To conclude, we show that after $O(\log n)$ additional steps the largest component has size $O(\log n)$ with probability $\Omega(1)$. If $L_1(X_T) \le (\ts-\varepsilon)n$ and ${\cal L}(X_T)$ is activated, then the definition of $\ts$ implies that the percolation step of the CM dynamics is sub-critical, and thus $X_{T+1}$ has no large component w.h.p. Hence, $X_{T+1}$ has no large component with probability $\Omega(1)$. Now, by Fact \ref{no-large-comp}(iii) and a union bound, all the new components created during the $O(\log n)$ steps immediately after time $T+1$ have size $O(\log n)$ w.h.p. Another union bound over components shows that during these $O(\log n)$ steps, every component in $X_{T+1}$ is activated w.h.p. Thus, after $O(\log n)$ steps the largest component in the configuration has size $O(\log n)$ with probability $\Omega(1)$, which establishes Lemma \ref{sub-contraction}.\end{proof}

\subsection{Convergence to typical configurations: the super-critical case}\label{section-sup}

\begin{lemma} \label{final-drift} Let $\lambda > \lambda_S$ and $\Delta_t := |L_1(X_t) - \theta_r n|$. If $X_0$ has at most one large component, then for some $T=O(\log n)$ there exists a constant $c > 0$ such that $\Pr[\,\Delta_T > A \sqrt{c n}\,] < 1/A$ for all $A \!>\! 0$.
\end{lemma}
\noindent
Let $\xi(r) = \sqrt{n r \log n}$, $\tS := 1 - q/\lambda$ and $\mu_t = L_1(X_t)+\frac{n-L_1(X_t)}{q}$. The following facts, which we prove later, will be used in the proof. 
\begin{fact} \label{sup-first-drift} If $X_0$ has at most one large component, then there exists $T=O(\log n)$ such that with probability $\Omega(1)$: $L_1(X_T)> (\tS+\varepsilon)n$, $L_2(X_T) = O(\log n)$ and $\sum_{j \ge 2} L_j(X_T)^2 = O(n)$. Moreover, once these properties are obtained they are preserved for a further $T' = O(\log  n)$ CM steps w.h.p.
\end{fact}
\begin{fact} \label{sup-no-large-comp} Assume $X_t$ has exactly one large component and all its other components have size at most $r < 2 n^{11/12}$. Then, for a small constant $\varepsilon > 0$ and sufficiently large $n$, each of the following holds with probability $1 - O\left(n^{-1}\right)$:
\begin{enumerate}[(i)]
	\item If ${\cal L}(X_t)$ is inactive and $L_1(X_t) > (\tS + \varepsilon)n$, then $L_1(X_{t+1}) = O(\log n)$.
	
	\item If ${\cal L}(X_t)$ is active, then $A_t \in J_{t,r} := \left[\mu_t-\xi(r),\mu_t+\xi(r)\right]$ and $G(A_t,p)$ is a super-critical random graph.
\end{enumerate}
\end{fact}
\begin{proof}[Proof of Lemma \ref{final-drift}] 
We show that one step of the CM dynamics contracts $\Delta_t$ in expectation. Observe that by Fact \ref{sup-first-drift} we may assume $X_0$ is such that $L_1(X_0)> (\tS+\varepsilon)n$, $L_2(X_0) = O(\log n)$ and $\sum_{j \ge 2} L_j(X_0)^2 = O(n)$, and that $\{X_t\}$ retains these properties for the $O(\log n)$ steps of this phase w.h.p. Consequently, if ${\cal L}(X_t)$ is inactive, then $L_1(X_{t+1}) = L_1(X_t)$ with probability $1-O(n^{-1})$ by Fact \ref{sup-no-large-comp}(i). Hence,
\begin{equation}
\E[\Delta_{t+1} \mid X_t,\neg{\cal E}_t] \le \E\left[\,\lvert L_1(X_{t+1})-L_1(X_{t})\rvert \mid X_t,\neg{\cal E}_t\right] + |L_1(X_{t})-\theta_r n| \le \Delta_t + O(1).  \label{sup:inactive-bound}
\end{equation}

To bound $\E[\Delta_{t+1} \,|\, X_t,{\cal E}_t]$, let $M_t = A_t - \mu_t$ and let $\ell_t(m)$ denote the size of the largest component of a $G(\mu_t+m,p)$ random graph. Also, let $\Delta_{t+1}' := |L_1(X_{t+1})-\phi(\theta_t) n|$. Note that, conditioned on $M_t = m$, $L_1(X_{t+1})$ and $\ell_t(m)$ have the same distribution. Moreover, if $A_t \in J_{t,r}$ then $M_t \in J_{t,r}' := [-\xi(r),\xi(r)]$. Hence, Fact \ref{sup-no-large-comp}(ii) with $r=O(\log n)$ implies 
\begin{eqnarray}
E[\Delta_{t+1}' \mid X_{t},{\cal E}_t] 
&\le& \sum_{m \in J_{t,r}'} \E[\Delta_{t+1}'\mid X_t,{\cal E}_t,M_t = m]~\Pr[M_t=m \mid X_t,{\cal E}_t] + O(1) \nonumber\\
&=& \sum_{m \in J_{t,r}'} \E[|\ell_t(m)-\phi(\theta_t) n|]\,\Pr[M_t=m \mid X_t,{\cal E}_t] + O(1). \nonumber
\end{eqnarray}
Now, by Fact \ref{sup-no-large-comp}(ii), $G(\mu_t+m,p)$ is a super-critical random graph, and thus $\E[|\ell_t(m)-\phi(\theta_t) n|] \le |m| + O(\sqrt{n})$  by Corollary \ref{cor:rg-sup}. Hence, 
\begin{equation}
E[\Delta_{t+1}' \mid X_{t},{\cal E}_t] \le \E[|M_t| \mid X_t,{\cal E}_t] + O(\sqrt{n}). \nonumber
\end{equation}
The following fact, which we prove later, follows straightforwardly from Hoeffding's inequality since $\sum_{j \ge 2} L_j(X_t)^2 = O(n)$. 

\smallskip
\begin{fact}\label{claim:exp-m} 
	$\E[\,|M_t|~|X_t,{\cal E}_t] = O(\sqrt{n})$.
\end{fact}
\smallskip

\noindent
Hence, 
\begin{equation}\label{equation:sup-crit:d1}
E[\Delta_{t+1}' \,|\, X_{t},{\cal E}_t] = O(\sqrt{n}),
\end{equation}
and the triangle inequality implies
\begin{equation}
\E[\Delta_{t+1}\mid X_{t},{\cal E}_t] \,\le\, \E[\Delta'_{t+1} \mid X_{t},{\cal E}_t] + |\theta_r-\phi(\theta_t)|n
\,\le\, |\theta_r - \phi(\theta_t)|n + O(\sqrt{n}). \label{bound:sc-exp-active-upper}
\end{equation} 
Putting (\ref{sup:inactive-bound}) and (\ref{bound:sc-exp-active-upper}) together, we have
\begin{equation}
\E[\Delta_{t+1}~|~X_t] \le \left(1-\frac{1}{q}\right)\Delta_t + \frac{|\theta_r - \phi(\theta_t)|n}{q} + O(\sqrt{n}). \nonumber
\end{equation}
By Lemma \ref{drift-bounds}(iii), there exists a constant $\delta \in (0,1)$ such that $\delta|\theta_t-\theta_r| \le |\theta_t-\phi(\theta_t)|$. Together with Lemma \ref{drift-bounds}(ii), this implies $|\theta_r-\phi(\theta_t)| \le (1-\delta)|\theta_t-\theta_r|$. Thus, there exists a constant $\delta' > 0$ such that 
\[\E[\Delta_{t+1}~|~ X_t] \le (1-\delta')\Delta_t  + \xi\]
where $\xi = O(\sqrt{n}$). Inducting,
\[\E[\Delta_t] \le (1-\delta')^t \Delta_0 + \xi/\delta'.\]
Hence, for some $t = O(\log n)$, $\E[\Delta_t] = O(\sqrt{n})$ and so Markov's inequality implies 
\[\Pr\left[\Delta_t > A \sqrt{c n}\right] \le 1/A\]
for some constant $c>0$ and any $A > 0$, which concludes the proof of Lemma \ref{final-drift}. \end{proof}

\noindent
We conclude this section with the proofs of Facts \ref{sup-first-drift}, \ref{sup-no-large-comp} and \ref{claim:exp-m}.

\begin{proof}[Proof of Fact \ref{sup-first-drift}:] 
	
This proof is similar to that of Fact \ref{no-large-comp}. If ${\cal L}(X_t)$ is inactive, then Hoeffding's inequality implies that $L_1(X_t) \le \frac{n-L_1(X_t)}{q} + \xi(r)$ with probability $1-O(n^{-2})$. Thus, $G(A_t,p)$ is a sub-critical random graph for constant $\varepsilon > 0$ since $q^{-1}(1-\tS)\lambda = 1$; part (i) then follows from Lemma \ref{rg-sub-1}. 

Part(ii) follows in similar fashion. If ${\cal L}(X_t)$ is active, then $L_1(X_t) \in J_{t,r}$ with probability $1-O(n^{-2})$ by Hoeffding's inequality. Hence, $\frac{\lambda}{n}(\mu_t - \xi(r)) > \frac{\lambda}{q}-o(1) > 1$ for sufficiently large $n$, which implies part (ii). 
\end{proof}

\begin{proof}[Proof of Fact \ref{claim:exp-m}:] Let $W_t$ be a random variable distributed according to the conditional distribution of $|M_t|$ given $X_t$ and ${\cal E}_t$. Since $\sum_{j \ge 2} L_j(X_t)^2 = O(n)$, Hoeffding's inequality implies that there exists a constant $c$ such that $\Pr[W_t > a\sqrt{n}] \le 2 \exp (- c a^2)$ for every $a>0$. Observe also that 
	\[W_t = \sum_{k=0}^n \1(W_t \ge k+1) + \1(k+1 > W_t > k)(W_t-k).\] Therefore,
	\[\E[W_t] \le \sum\limits_{k=0}^n \Pr[W_t > k] \le 1 + 2 \sum\limits_{k=1}^n e^{- \frac{c k^2}{n}} \le 1 + 2\int\limits_{0}^{\infty} e^{- \frac{c x^2}{n}} dx = O(\sqrt{n}),\]
	as desired.\end{proof}

\begin{proof}[Proof of Fact \ref{sup-first-drift}:] Let $d_t := (\tS + 2\varepsilon) n - L_1(X_t)$. Then,
	\begin{eqnarray}
	\E[d_{t+1} \mid X_t,{\cal E}_t] = d_t + \theta_t n - \E[L_1(X_{t+1})\;|\;X_t,{\cal E}_t]. \label{sup-d-drift}
	\end{eqnarray}
	
	Let $h^-(\theta_t) = \theta_t n + (1-\theta_t)q^{-1} n - \xi(r)$ and let $\ell^-(\theta_t)$ be a random variable distributed as the size of the largest component of a $G(h^-(\theta_t),p)$ random graph. If ${\cal L}(X_t)$ is activated, Fact \ref{sup-no-large-comp}(ii) implies that $A_t \in J_{t,r}$ with probability $1 - O\left(n^{-1}\right)$ where $r < 2 n^{11/12}$. Therefore,
	\begin{eqnarray}
	\E[L_1(X_{t+1})\;|\;X_t,{\cal E}_t] &\ge& \sum_{a \in J_{t,r}} \E[L_1(X_{t+1})\;|\;X_t,{\cal E}_t,A_t=a] \Pr[A_t = a\;|\;X_t,{\cal E}_t] \nonumber\\
	&\ge& \E[L_1(X_{t+1})\;|\;X_t,{\cal E}_t,A_t=h^-(\theta_t)] - \Omega(1)
	= \E[\ell^-(\theta_t)] - \Omega(1) \nonumber
	\end{eqnarray}
	Note that $G(h^-(\theta_t),p)$ is a super-critical random graph since $\lambda > q$. Hence, Corollary \ref{cor:rg-sup} implies that $\E[L_1(X_{t+1})\;|\;X_t,{\cal E}_t] \ge \phi(\theta_t) n - \Omega(\xi(r))$. Plugging this bound into (\ref{sup-d-drift}), we have
	\begin{eqnarray}
	\E[d_{t+1}\;|\;X_t,{\cal E}_t] &\le& d_t + (\theta_t - \phi(\theta_t))n + O(\xi(r)). \nonumber
	\end{eqnarray}
	
	Now, by Fact \ref{drift-bounds-1}, $\theta_r > \tS$. Thus, when $d_t > 0$, Lemma \ref{drift-bounds} implies that there exists a constant $\delta \in (0,1)$ such that $\phi(\theta_t) - \theta_t > \delta(\theta_r-\theta_t) > \delta(\theta_r-\tS-2\varepsilon) = \delta'$, where $\delta'$ is a constant in $(0,1)$ for a sufficiently small $\varepsilon$. Moreover, $\xi(r) = o(n)$ and thus 
	\begin{equation}
	\E[d_{t+1}-d_t \mid X_t,{\cal E}_t,d_t>0] \le -\delta' n +  O(\xi(r)) \le -\gamma n, \label{super-critical:claim:drift}
	\end{equation}	
	for some constant $\gamma > 0$.
	
	Assuming $d_0 >0$, let $\tau = \min\{t > 0: d_t \le 0 \}$ and let ${\cal H}_K$ be the event that ${\cal L}(X_t)$ is activated for all $t \in [0,K]$, where $K$ is a fixed constant we choose later. Let $\hat{T} := \min\{\tau,K\}$ and observe that conditioned on ${\cal H}_K$, (\ref{super-critical:claim:drift}) holds for all $t \le \hat{T}$. Hence, Lemma \ref{lemma:hitting-times-sub} implies $\E[\hat{T}|{\cal H}_K] \le 4/\gamma$, and by Markov's inequality we have 
	\[\Pr[\tau \le K/2 \mid {\cal H}_K ] \ge  \Pr[\hat{T} \le K/2 \mid {\cal H}_K] \ge 1 - \frac{8}{\gamma K}.\]
	Since the event ${\cal H}_K$  occurs with constant probability $q^{-K}$, we have $L_1(X_T) \ge (\tS+2\varepsilon)n$ with probability $\Omega(1)$ for some $T=O(1)$.
	
	We now show that if $L_1(X_0) > (\tS+\varepsilon)n$, then $L_1(X_1) > (\tS+\varepsilon)n$ with probability $1-O(n^{-1})$. A union bound then implies $L_1(X_t) > (\tS+\varepsilon)n$ for all $t \in [0,T]$ with probability $1-O\left(T/n\right)$. If ${\cal L}(X_0)$ is not activated, by Fact \ref{sup-no-large-comp}(i), $L_1(X_1) = L_1(X_0) > (\tS+\varepsilon)n$ with probability $1-O\left(n^{-1}\right)$. Otherwise, if ${\cal L}(X_0)$ is activated, Fact \ref{sup-no-large-comp}(ii) implies that $A_0 \in J_{0,r}$ with probability $1-O(n^{-1})$. Conditioning on $A_0 \in J_{0,r}$, $L_1(X_1) \succeq \ell^-(\theta_0)$ and by Lemma \ref{lemma:rg-sup},
	\[\Pr[L_1(X_1) < \phi(\theta_0)n-2\xi(r)|A_0 \in J_{0,r}] \le \Pr[\ell^-(\theta_0) < \phi(\theta_0)n-2\xi(r)] = O(n^{-1}). \]
	Lemma \ref{drift-bounds} and Fact \ref{drift-bounds-1} imply $\phi(\theta_0) n-2\xi(r) > (\tS+\varepsilon)n$ for sufficiently large $n$ since $\xi(r) = o(n)$. Hence, $L_1(X_1) > (\tS+\varepsilon)n$ with probability $1-O\left(n^{-1}\right)$. This concludes the proof of the first part of Fact \ref{sup-first-drift}.
	
	We show next that $L_2(X_T) = O(\log n)$ w.h.p.\ for some $T=O(\log n)$. For this, we condition on $L_1(X_t) > (\tS+\varepsilon)n$ for $t \in [0,T]$ with $T=O(\log n)$. Then Fact \ref{sup-no-large-comp}, Lemma \ref{lemma:rg-sup} and a union bound imply that every new small component has size $O(\log n)$ with probability $1-O\left(T/n\right)$. The probability that any initial component remains after $T=B \log n$ steps is $O(n^{-1})$ for a sufficiently large constant $B > 0$; therefore, $L_2(X_T) = O(\log n)$ with probability $1 - O\left(\log n/n\right)$. Fact \ref{sup-no-large-comp} and another union bound implies that this property is maintained for an additional $O(\log n)$ steps w.h.p.
	
	Finally, we show that $\sum_{j \ge 2} L_j(X_T)^2 = O(n)$ for some $T=O(\log n)$. Consider the one-dimensional random process $\{Z_t\}$ where $Z_t = \sum_{j \ge 2} L_j(X_t)^2$. At time $t$, the decrease in $Z_t$ as a result of the dissolution of active components is $Z_t/q$ in expectation, and is at least $Z_t/q - o(n)$ with probability $1-O(n^{-1})$ by Hoeffding's inequality. Lemma \ref{rg-suscept} implies that the increase in $Z_t$ as a result of the creation of new components in the percolation step is at most $Cn$ with probability $1-O(n^{-1})$. Therefore,
	\[\E[Z_{t+1}~|~X_t] \le Z_t - \frac{Z_t}{q}+C n + o(n) \le \left(1-\frac{1}{3q}\right)Z_t\]
	provided $Z_t \ge 3 C q n$. Thus, Markov's inequality ensures $Z_T < 3 C q n$ with probability $\Omega(1)$ for some $T=O(\log n)$. Finally, when $Z_t > 3 C q n$, $Z_t$ decreases by at least $3 C n - o(n)$ with probability $1-O(n^{-1})$; therefore, $Z_{t+1} \le Z_t$ with probability $1-O(n^{-1})$. When $Z_t \le 3 C q n$, $Z_{t+1} \le (3q+1)Cn$ with probability $1-O(n^{-1})$. Hence, if $Z_0 \le 3 C q n$, then $Z_t=O(n)$ for $t \in [0,T]$ with $T=O(\log n)$ w.h.p.\end{proof}

\subsection{Coupling to the same component structure}\label{upper-sc}

In this section we design a coupling of the CM steps which, starting from two configurations with certain properties (namely those obtained in Sections \ref{upper-sub} and \ref{section-sup} for the sub-critical and super-critical case respectively), quickly converges to a pair of configurations with the same component structure. (We say that two random-cluster configurations $X$ and $Y$ have the same component structure if $L_j(X)=L_j(Y)$ for all $j \ge 1$.) 

The only additional property we will require is that the starting configurations should have a linear number of isolated vertices. Although in Sections \ref{upper-sub} and \ref{section-sup} we do not guarantee this, observe that in the sub-critical (resp., super-critical) case, Fact \ref{no-large-comp} (resp., Fact \ref{sup-no-large-comp}) and Lemma \ref{rg-isolated} imply that a single CM step from a configuration with a unique large component produces a configuration with a linear number of isolated vertices w.h.p. 

We will focus first on the super-critical case, since a simplified version of the arguments works in the sub-critical case. 

\begin{lemma}\label{coupling-scs} Let $\lambda > q$ and let $X_0$ and $Y_0$ be random-cluster configurations such that:
	\begin{enumerate}[(i)]
		\item $I(X_0)$, $I(Y_0) = \Omega(n)$;
		\item $L_2(X_0)$, $L_2(Y_0) = O(\log n)$; 
		\item $|L_1(X_0)-\theta_r n|$, $|L_1(Y_0)-\theta_r n| = O(\sqrt{n}\log^2 n)$; and
		\item $\sum_{j \ge 2} L_j(X_0)^2$, $\sum_{j \ge 2} L_j(Y_0)^2 = O(n)$.
	\end{enumerate}
	Then, there exists a coupling of the CM steps such that $X_T$ and $Y_T$ have the same component structure after $T = O(\log n)$ steps with probability $\Omega(1)$.
\end{lemma}

\begin{proof} First we make certain that properties (i) to (iv) are preserved throughout this phase w.h.p. By Fact \ref{sup-first-drift}, (ii) and (iv) are maintained w.h.p.\ for $O(\log n)$ steps. Also, it follows from Lemma \ref{final-drift} (with $A = O(\log ^2 n)$) and a union bound that (iii) is preserved for $O(\log n)$ steps w.h.p. Finally, by Fact \ref{sup-no-large-comp}, if a configuration has properties (ii) and (iii), then the number of active vertices is $\Omega(n)$ with probability $1-O(n^{-1})$; Lemma \ref{rg-isolated} and a union bound then imply that (i) is preserved w.h.p.\ for $O(\log n)$ steps. Hence, we can assume that these properties are maintained throughout the $O(\log n)$ steps of this phase.
	
	
	Our coupling will be a composition of three couplings. Coupling I contracts a certain notion of distance between $\{X_t\}$ and $\{Y_t\}$. This contraction will boost the probability of success of the other two couplings. Coupling II is a one-step coupling which guarantees that the largest components from $\{X_t\}$ and $\{Y_t\}$ have the same size with probability $\Omega(1)$. Coupling III uses the binomial coupling from Lemma \ref{bc-str} to achieve two configurations with the same component structure with probability $\Omega(1)$. 
	
	{\bf Coupling I:}
	Excluding ${\cal L}(X_t)$ and ${\cal L}(Y_t)$, consider a maximal matching $W_t$ between the components of $X_t$ and $Y_t$ with the restriction that only components of equal size are matched to each other. Let $M(X_t)$ and $M(Y_t)$ be the components in the matching from $X_t$ and $Y_t$ respectively. Let $D(X_t)$ and $D(Y_t)$ be the complements of ${\cal L}(X_t) \cup M(X_t)$ and ${\cal L}(Y_t) \cup M(Y_t)$ respectively, and let $d_t = |D(X_t)| + |D(Y_t)|$ where $|\cdot|$ denotes the total number of vertices in the respective components.
	
	The activation of the components in $M(X_t)$ and $M(Y_t)$ is coupled using the matching $W_t$. That is, $c \in M(X_t)$ and $W_t(c) \in M(Y_t)$ are activated simultaneously with probability $1/q$. The activations of ${\cal L}(X_t)$ and ${\cal L}(Y_t)$ are also coupled, and the components in $D(X_t)$ and $D(Y_t)$ are activated independently. Let $A(X_t)$ and $A(Y_t)$ denote the set of active {\it vertices} in $X_t$ and $Y_t$ respectively, and w.l.o.g. assume $|A(X_t)| \ge |A(Y_t)|.$ Let $R_t$ be an arbitrary subset of $A(X_t)$ such that $|R_t|=|A(Y_t)|$ and let $Q_t=A(X_t)\setminus R_t$. The percolation step is coupled by establishing an arbitrary vertex bijection $b_t: R_t \rightarrow A(Y_t)$ and coupling the re-sampling of each edge $(u,v) \in R_t \times R_t$ with $(b_t(u),b_t(v)) \in A(Y_t) \times A(Y_t).$ Edges within $Q_t$ and in the cut $C_t = R_t \times Q_t$ are re-sampled independently. The following claim establishes the desired contraction in $d_t$.
	
		\begin{claim}\label{coup-I} Let $\omega(n) = n/\log^4 n$; after $T = O(\log \log n)$ steps, $d_{T} \le \omega(n)$ w.h.p.
		\end{claim}
		
		\begin{proof} Let $D_a(X_t)$ and $D_a(Y_t)$ be the number of active vertices from $D(X_t)$ and $D(Y_t)$ respectively, and let ${\cal F}_t$ be the history of the first $t$ steps. Observe that Coupling I guarantees that $R_t$ and $A(Y_t)$ will have the same component structure internally. However, the vertices in $Q_t$ will contribute to $d_{t+1}$ unless they are part of the new large component, and each edge in $C_t$ could increase $d_{t+1}$ by at most (twice) the size of one component of $R_t$, which is $O(\log n)$. Thus,
			\begin{equation}\label{exp-contract}
			\E[d_{t+1}\,|\,A(X_t),A(Y_t),C_t,{\cal F}_t] \le d_t - (|D_a(X_t)|+|D_a(Y_t)|) + |Q_t| + 2|C_t| \times O(\log n). 
			\end{equation}
			Observe that $\E[|D_a(X_t)|+|D_a(Y_t)|~|{\cal F}_t] = d_t/q$, and $\E[|C_t|~|A(X_t),A(Y_t),{\cal F}_t] = |R_t||Q_t| p \le \lambda |Q_t|$. Since $|Q_t| = O(\sqrt{n}\log^2 n)$, taking expectations in (\ref{exp-contract}) we get
			\begin{equation}
			\E[d_{t+1}\,|\,{\cal F}_t] \le d_t - \frac{d_t}{q} + O\left( \sqrt{n} \log^3 n\right) \le \left(1 - \frac{1}{2q}\right)d_t \nonumber
			\end{equation}
			provided $d_t > \omega(n)$. Thus, Markov's inequality implies $d_{T} \le \omega(n)$ for some $T=O(\log \log n)$ w.h.p. Note that for larger values of $T$, this argument immediately provides stronger bounds for $d_T$, but neither our analysis nor the order of the coupling time benefits from this. \end{proof}
		
		{\bf Coupling II:} Assume now that $d_t \le \omega(n)$ and let $I_{\textsc{m}}(X_t)$ and $I_{\textsc{m}}(Y_t)$ denote the isolated vertices in $M(X_t)$ and $M(Y_t)$ respectively. The activation in $X_t \setminus I_{\textsc{m}}(X_t)$ and $Y_t \setminus I_{\textsc{m}}(Y_t)$ is coupled as in Coupling I, except we condition on the event that ${\cal L}(X_t)$ and ${\cal L}(Y_t)$ are activated, which occurs with probability $1/q$. This first part of the activation could activate a different number of vertices from each copy of the chain; let $\rho_t$ be this difference. 
		
		First we show that $\rho_t = O(\sqrt{n})$ with probability $\Omega(1)$. By Lemma \ref{final-drift} (with $A=2$), we have $|L_1(X_t)-L_1(Y_t)| \!=\! O(\sqrt{n})$ with probability $\Omega(1)$. If this is the case, then $||D(X_t)|-|D(Y_t)|| = O(\sqrt{n})$. Also, since $\sum_{j \ge 2} L_j(X_t)^2 \!=\! O(n)$ and $\sum_{j \ge 2} L_j(Y_t)^2 \!=\! O(n)$, by Hoeffding's inequality the numbers of active vertices from $D(X_t)$ and $D(Y_t)$ differ by at most $O(\sqrt{n})$ with probability $\Omega(1)$. Thus, $\rho_t \!=\! O(\sqrt{n})$ with probability $\Omega(1)$.
		
		Now we show how to couple the activation in $I_{\textsc{m}}(X_t)$, $I_{\textsc{m}}(Y_t)$ in a way such that $|A(X_t)|\!=\!|A(Y_t)|$ with probability $\Omega(1)$. The number of active isolated vertices from $I_{\textsc{m}}(X_t)$ is binomially distributed with parameters $|I_{\textsc{m}}(X_t)|$ and $1/q$, and similarly for $I_{\textsc{m}}(Y_t)$. Hence, the activation of the isolated vertices may be coupled using the binomial coupling from Section \ref{bc}. Since $|I_{\textsc{m}}(X_t)| \!=\!|I_{\textsc{m}}(Y_t)|\!=\! \Omega(n)$ and $\rho_t \!=\! O(\sqrt{n})$, Lemma \ref{bc-str} 
		implies that this coupling corrects the difference $\rho_t$ with probability $\Omega(1)$. If this is the case, then by coupling the edge sampling bijectively as in Coupling I, we ensure that $L_1(X_{t+1}) = L_1(Y_{t+1})$ and $d_{t+1} \le \omega(n)$ with probability $\Omega(1)$. 
		
		{\bf Coupling III:} Assume $L_1(X_0) \!=\! L_1(Y_0)$ and $d_0 \le \omega(n)$. The component activation is coupled as in Coupling II, but we do not require the two large components to be active; rather, we just couple their activation together. 
		
		If $L_1(X_t) = L_1(Y_t)$, then $|D(X_t)|=|D(Y_t)|$ and thus the expected number of active vertices from $D(X_t)$ and $D(Y_t)$ is the same. Consequently, since $d_t \le \omega(n)$, Hoeffding's inequality implies $\rho_t = O\left(\sqrt{n}\log^{-1} n\right)$ w.h.p. Let ${\cal F}_t$ be the event that the coupling of the isolated vertices succeeds in correcting the error $\rho_t$. Since $|I_{\textsc{m}}(X_t)|=|I_{\textsc{m}}(Y_t)|=\Omega(n)$, ${\cal F}_t$ occurs with probability $1-O(\log^{-1} n)$ by Lemma \ref{bc-str}. If this is the case, the updated part of both configurations will have the same component structure; thus, $L_1(X_{t+1}) = L_1(Y_{t+1})$ and $d_{t+1} \le d_t$. Hence, if ${\cal F}_t$ occurs for all $0 \le t\le T$, then $X_T$ and $Y_T$ fail to have the same component structure only if at least one of the initial components was never activated. For $T=O(\log n)$ this occurs with at most constant probability. Since ${\cal F}_t$ occurs for $T=O(\log n)$ consecutive steps with at least constant probability, then $X_T$ and $Y_T$ have the same component structure with probability $\Omega(1)$. 
		
		Couplings I, II and III succeed each with at least constant probability. Thus, the overall coupling succeeds with probability $\Omega(1)$, as desired.\end{proof}
	
	\noindent
	In the sub-critical case we may assume also that $L_1(X_0)$ and $L_1(Y_0)$ are $O(\log n)$. Therefore, a simplified version of the same coupling works since Coupling II is not necessary.
	
	\begin{cor}\label{coupling-scs-1} Suppose $\lambda < \ls$ and $X_0$ and $Y_0$ are as in Lemma \ref{coupling-scs}. Then, there exists a coupling of the CM steps such that $X_T$ and $Y_T$ have the same component structure with probability $\Omega(1)$, for some $T=O(\log n)$.
	\end{cor}

\subsection{Coupling to the same configuration}\label{sub-section:phase-four}

\begin{lemma}\label{coupling-ec} Let $X_0$, $Y_0$ be two random-cluster configurations with the same component structure. Then, there exists a coupling of the CM steps such that after $T=O(\log n)$ steps $X_T=Y_T$ w.h.p.
\end{lemma}

\begin{proof}
	Let $B_t$ a bijection  between the vertices of $X_t$ and $Y_t$. We first describe how to construct $B_0$. Consider a maximal matching between the  components of $X_0$ and $Y_0$ with the restriction that only components of equal size are matched to  each other. Since the two configurations have the same component structure all components are matched. Using this matching, vertices between matched components are mapped arbitrarily to obtain $B_0$.
	
	Vertices mapped to themselves we call ``fixed''. At time $t$, the component activation is coupled according to $B_t$. That is, if $B_t(u) = v$ for $u \in X_t$ and $v \in Y_t$, then the components containing $u$ and $v$ are simultaneously activated with probability $1/q$. $B_{t+1}$ is adjusted such that if a vertex $w$ becomes active in both configurations then $B_{t+1}(w)=w$; the rest of the activated vertices are mapped arbitrarily in $B_{t+1}$ and the inactive vertices are mapped like in $B_t$. The percolation step at time $t$ is then coupled using $B_{t+1}$. That is, the re-sampling of the active edge $(u,v) \in X_t$ is coupled with the re-sampling of the active edge $(B_{t+1}(u),B_{t+1}(v)) \in Y_t$.
	
	This coupling ensures that the component structures of $X_t$ and $Y_t$ remain the same for all $t \ge 0$. Moreover, once a vertex is fixed it remains fixed forever. The probability that a vertex is fixed in one step is $1/q^2.$ Therefore, after $O(\log n)$ steps the probability that a vertex is not fixed is at most $1/n^2$. A union bound over all vertices implies that $X_T = Y_T$ w.h.p.\ after $T = O(\log n)$ steps.
\end{proof}

\section{Mixing time lower bounds}\label{sec:lower-bounds}

In this section we prove the exponential lower bound on the mixing time of the CM dynamics for $\lambda$ in the critical window $(\ls,\lS)$, as stated in Theorem \ref{thm:intro2} in the introduction. We also prove a $\Omega(\log n)$ lower bound in the ``fast mixing'' regime, showing that our upper bounds in Section \ref{sec:upper-bounds} are tight. 

Recall from the introduction that when $q=2$ and $\lambda < \ls = \lambda_c$, the SW dynamics mixes in $\Theta(1)$ steps and thus the CM dynamics requires $\Theta(\log n)$ additional steps to mix. This is due to the fact that the CM dynamics may require as many steps to activate all the components from the initial configuration.

We will reuse some notation from the previous sections. Namely, we will use $A_t$ for the number of activated vertices in sub-step {\rm (i)} of the CM dynamics at time $t$, ${\cal L}(X_t)$ for the largest component in $X_t$ and $L_i(X_t)$ for the size of the $i$-th largest component. We will also write $\theta_t n$ for $L_1(X_t)$ and use ${\cal E}_t$ for the event that ${\cal L}(X_t)$ is activated.

\begin{thm}\label{lb} For any $q > 1$, the mixing time of the CM dynamics is $\exp(\Omega(\sqrt{n}))$ for $\lambda \in (\lambda_s,\lambda_S)$, and $\Omega(\log n)$ for $\lambda \not\in [\lambda_s,\lambda_S]$.
\end{thm}

\begin{proof} The random-cluster model undergoes a phase transition at $\lambda_c$, so it is natural to divide the proof into four cases: $\lambda < \ls$, $\lambda \in (\ls,\lambda_c)$, $\lambda \in [\lambda_c,\lS)$ and $\lambda > \lS$. Note that when $q \le 2$ the interval $(\ls,\lS)$ is empty and the exponential lower bound is vacuously true.
	
	{\bf Case (i):} $\lambda < \ls$.\ \  Let $X_0$ be a configuration where all the components have size $\Theta(\log^2 n)$ and let $b = q/(q-1).$ The probability that a particular component is not activated in any of the first $T = \frac{1}{2}\log_b n$ steps is $(1-1/q)^T = n^{-1/2}$. Therefore, the probability that all initial components are activated in the first $T$ steps is $(1- n^{-1/2})^K$ with $K = \Theta(n/\log^2 n)$. Thus after $T$ steps $L_1(X_T) = \Theta(\log^2 n)$ w.h.p. and the result follows from Lemma \ref{rc-sub}.
	
	{\bf Case (ii):} $q > 2$ and $\lambda_c \le \lambda < \lS = q$.\ \ The idea for this bound comes from \cite{GJ}. Let $S$ be the set of graphs $G$ such that $L_1(G) = \Theta(\sqrt{n})$ and let $X_0 \in S$.  Let $\mu := \E[A_0] = n/q$; then by Hoeffding's inequality 
	$\Pr\left[\left|A_0-\mu\right| > \varepsilon n\right] \le 2\exp\left(-2\varepsilon^2\sqrt{n}\right)$.
	If $A_0 < \mu+\varepsilon n,$ the active subgraph is sub-critical for sufficiently small $\varepsilon$. Therefore, Lemma \ref{sub-strong} implies that
	$\Pr[X_1 \not\in S|X_0\in S] \le e^{-c \sqrt{n}}$
	for some constant $c>0$. Hence, $\Pr[X_1,...,X_t \in S|X_0 \in S] \ge 1 - t e^{-c\sqrt{n}} \ge 3/4$ for $t = \lfloor e^{c\sqrt{n}}/4 \rfloor$. The result again follows from Lemma \ref{rc-sup}.
	
	{\bf Case (iii):} $q > 2$ and $\lambda_s < \lambda < \lambda_c$.\ \ The intuition for this case comes directly from Figure \ref{fig:f}. In this regime, Fact \ref{f-roots} implies that the function $f(\vf) = \vf - \phi(\vf)$ has two positive zeros $\vf^*$ and $\theta_r$ in $(\theta_{\rm min},1]$ with $\vf^* < \theta_r$. Moreover, $f$ is negative in the interval $(\vf^*,\theta_r)$. Therefore, any configuration with a unique large component of size $\theta n$ with $\theta \in (\vf^*,\theta_r)$ will ``drift" towards a configuration with a bigger large component. However, a typical random-cluster configuration in this regime does not have a large component. This drift in the incorrect direction is sufficient to prove the exponential lower bound in this regime. We now proceed to formalize this intuition.
	
	Let $S$ be the set of graphs $G$ such that $L_1(G) > (\vf^*+\varepsilon) n$ and $L_2(G) = O(\sqrt{n})$ where $\varepsilon$ is a small positive constant to be chosen later. Assume $X_0 \in S$. If ${\cal L}(X_0)$ is inactive, by Hoeffding's inequality $A_0 \in I_0 := \left[(1-\theta_0)n/q-\gamma_0 n,(1-\theta_0)n/q+\gamma_0 n\right]$ with probability $1 - e^{-\Omega(\sqrt{n})}$ for any desired constant $\gamma_0 > 0$. If $A_0 \in I_0$, then, for a sufficiently small $\gamma_0$, the percolation step is sub-critical, and by Lemma \ref{sub-strong}, $\Pr[X_1 \not\in S|X_0 \in S] = e^{-\Omega(\sqrt{n})}.$
	
	When ${\cal L}(X_0)$ is active, we show that for any desired constant $\rho > 0$, $L_1(X_1) \in [\phi(\theta_0)n-\rho n,\phi(\theta_0)n+\rho n]$ with probability $1-e^{-\Omega(\sqrt{n})}$. Let $\mu_0 := \theta_0 n+(1-\theta_0)n/q$; then by Hoeffding's inequality $A_0 \in I_1 := [\mu_0-\gamma_1 n, \mu_0 +\gamma_1 n]$ with probability $1 - e^{-\Omega(\sqrt{n})}$ for any desired constant $\gamma_1 > 0$. Let $h(\theta_0) = \mu_0 n+\gamma_1 n$ and let $\ell(\theta_0)$ be a random variable distributed as the size of the largest component of a $G(h(\theta_0),p)$ random graph. Then, for any $\rho > 0$,
	\begin{eqnarray} 
	\Pr[L_1(X_1) > \phi(\theta_0)n+\rho n] &\le& \sum\limits_{a \in I_1} \Pr[L_1(X_1) > \phi(\theta_0)n+\rho n|A_0 = a]\Pr[A_0 = a] + e^{-\Omega(\sqrt{n})} \nonumber\\
	&\le& \Pr[\,\ell(\theta_0) > \phi(\theta_0)n+\rho n\,] + e^{-\Omega(\sqrt{n})}.\nonumber
	\end{eqnarray}
	
	Recall from Section \ref{subsection:drift} that when $\lambda < q$, $\lambda(\theta_{\rm min} + (1-\theta_{\rm min})q^{-1}) = 1$. Therefore, the $G(h(\theta_0),p)$ random graph is super-critical since $\theta_0 > \vf^* > \theta_{\rm min}$. Let $\beta = \beta(\lambda')$ with $\lambda' = \lambda h(\theta_0) / n$ where $\beta(\lambda')$ is defined in (\ref{rg-root-eq}). By Lemma \ref{sup-gc-ld}, $\ell(\theta_0) \in [\beta n - \gamma_2 n, \beta n + \gamma_2 n]$ with probability at least $1 - e^{-\Omega(n)}$ for any desired constant $\gamma_2 > 0$. Observe that if $\gamma_1 = 0$, then $\beta = \phi(\theta_0)$ by the definition of $\phi$. Then by continuity, for any constant $\delta > 0$ there exists $\gamma_1$ small enough such that $|\phi(\theta_0)-\beta| < \delta$. Thus, $\ell(\theta_0) \in [\phi(\theta_0) n - \rho n, \phi(\theta_0) n + \rho n]$ with probability $1 - e^{-\Omega(n)}$. Consequently, $\Pr[L_1(X_1) > \phi(\theta_0)n+\rho n] = e^{-\Omega(\sqrt{n})}$.  By a similar argument $\Pr[L_1(X_1) < \phi(\theta_0)n-\rho n] = e^{-\Omega(\sqrt{n})}$, and then $L_1(X_1) \in [\phi(\theta_0)n-\rho n,\phi(\theta_0)n+\rho n]$ with probability $1-e^{-\Omega(\sqrt{n})}.$ 
	
	Now we show that for suitable positive constants $\varepsilon$ and $\rho$, $\phi(\theta_0)-\rho > \vf^*+\varepsilon$; this implies $L_1(X_1) > (\vf^*+\varepsilon)n$ with probability $1-e^{-\Omega(\sqrt{n})}$. Note that Part $(ii)$ of Lemma \ref{drift-bounds} still holds when $\lambda > \lambda_s$ and $\theta \in (\vf^*,1)$. Hence, if $\theta_0 > \theta_r - \varepsilon$, then $\phi(\theta_0) > \theta_r - \varepsilon$. Therefore, we can choose $\varepsilon$ and $\rho$ such that $\phi(\theta_0)-\rho > \vf^*+\varepsilon$. If $\theta_0 < \theta_r - \varepsilon$, then $\phi(\theta_0) > \theta_0 > \vf^* + \varepsilon$ since $f$ is negative in this interval. Note that $\phi(\theta_0) - \theta_0 = -f(\theta_0)$, so in this case we can pick $\rho$ to be $-1/2$ of the maximum of $f$ in $[\vf^*+\varepsilon,\theta_r-\varepsilon]$ for a sufficiently small $\varepsilon$. Thus, $L_1(X_1) > (\vf^*+\varepsilon)n$ with probability $1-e^{-\Omega(\sqrt{n})}$. 
	
	By Lemma \ref{sub-strong}, $L_2(X_1) = O(\sqrt{n})$ with probability $1 - e^{-\Omega(\sqrt{n})}$. Hence, $\Pr[X_1 \not\in S|X_0 \in S] \le e^{-c\sqrt{n}}$ for some constant $c > 0$, and then $\Pr[X_1,...,X_t \in S|X_0 \in S] \ge 1 - t e^{-c\sqrt{n}} \ge 3/4$ for $t = \lfloor e^{c\sqrt{n}}/4 \rfloor$. The result then follows from Lemma \ref{rc-sup}.
	
	{\bf Case (iv):} $\lambda > \lS = q$.\ \ The idea for this bound comes from \cite{LNNP}. Let $\omega := q/(q-1)$ and let $\Delta_t := |L_1(X_t)-\theta_r n|$ as in Section \ref{section-sup}. We will show that $\Delta_{t+1} \ge \Delta_t / 2\omega$ w.h.p.\ provided $L_1(X_t)$ is sufficiently large and $L_2(X_t) = O(\log n)$. An inductive argument will then allow us to conclude that for a suitable starting configuration, the CM dynamics requires $\Omega(\log n)$ steps to shrink the size of the largest component to close to $\theta_r n$.
	
	We provide first some intuition on how we prove that $\Delta_{t+1} \ge \Delta_t / 2\omega$ w.h.p. Observe that if ${\cal L}(X_t)$ is inactive, we know from Section \ref{section-sup} that ${\cal L}(X_{t+1}) = {\cal L}(X_t)$ and $\Delta_{t+1}=\Delta_t$ w.h.p. When ${\cal L}(X_t)$ is active, we use a bound on the derivative of the function $f$ (as defined in Section \ref{subsection:drift}) in the interval $(\theta_r,1)$ to argue that $\phi(\theta_t)$ is close to $\theta_t$ (or, more precisely, that $\phi(\theta_t)-\theta_r$ is not much smaller than $\theta_t-\theta_r$). Hence, if $\theta_{t+1}$ is much closer to $\theta_r$ than $\theta_t$ (i.e., $\Delta_{t+1} < \Delta_t / 2\omega$), then $\theta_{t+1}$ will have to be far from $\phi(\theta_t)$, which we know from Section \ref{section-sup} is unlikely.  We now proceed to formalize this intuition.
	
	\begin{fact} \label{slow-drift} If $\theta \in (\theta_r,1]$, then $\phi(\theta)-\theta_r \ge \frac{\theta-\theta_r}{\omega}$.
	\end{fact}
	\begin{proof} By Fact \ref{drift-phi-prop}, $\phi'(\theta) > \frac{q-1}{q}$ and $\phi(\theta_r) = \theta_r$. Then, by the mean value theorem,
		\[\frac{q-1}{q} < \frac{\phi(\theta)-\phi(\theta_r)}{\theta-\theta_r} = \frac{\phi(\theta)-\theta_r}{\theta-\theta_r},\]
		and the result follows.\end{proof}
	
	\noindent
	We choose $X_0$ with $L_1(X_0) = \theta_0 n$ sufficiently large (namely, $\theta_0$ much larger than $\theta_r > \Theta_{\rm S} = 1- q/\lambda$) and $L_2(X_0) = O(\log n)$. Fact \ref{sup-first-drift} implies that the CM dynamics preserves these properties during $T = O(\log n)$ steps w.h.p., which allows us to assume that they are maintained throughout the $O(\log n)$ steps of this phase. Thus, by (\ref{equation:sup-crit:d1}), we have
	\begin{equation}
	\E[\,|\theta_{t+1} - \phi(\theta_t)|n ~|~{\cal E}_t\,] \le \gamma\sqrt{n} \label{eq:lb-exp}
	\end{equation}
	for any $t \le T$ and some constant $\gamma > 0$.
	
	Let $\Gamma_t$ be the event that $\theta_{t} n  > \theta_r n + n^\alpha/(2\omega)^t$ for some constant $\alpha > 0$ that we will choose later. Note that (\ref{eq:lb-exp}) still holds if we condition on $\Gamma_t$. Hence, Markov's inequality and Fact \ref{slow-drift} imply
	\[\Pr\left[\,|\theta_{t+1} - \phi(\theta_t)|n > \frac{1}{2}|\phi(\theta_t) -\theta_r|n ~\;\middle\vert\;~{\cal E}_t,\Gamma_t\,\right] \le \frac{2 \gamma \sqrt{n}}{|\phi(\theta_t)-\theta_r| n} \le \frac{\gamma (2\omega) \sqrt{n}}{|\theta_t-\theta_r| n}.\]
	Now, since $\theta_t > \theta_r$ and, by Lemma \ref{drift-bounds}, $\theta_t \ge \phi(\theta_t) \ge \theta_r$, we have 
	\[\Pr\left[\,|\theta_{t+1} - \phi(\theta_t)|n \le \frac{1}{2}|\phi(\theta_t) -\theta_r|n ~\;\middle\vert\;~{\cal E}_t,\Gamma_t\,\right]	\ge 1 - \frac{\gamma (2\omega)^{t+1}}{n^{\alpha-1/2}}\]
	Fact \ref{slow-drift} implies that if $\theta_{t+1} - \theta_r < \frac{\theta_t-\theta_r}{2\omega}$, then $\theta_{t+1} < \frac{\phi(\theta_t)+\theta_r}{2}$. Consequently,
	\[|\theta_{t+1}-\phi(\theta_t)| > (\phi(\theta_t)-\theta_{t+1}) + \left(\theta_{t+1}-\frac{\phi(\theta_t)+\theta_r}{2}\right) > \frac{\phi(\theta_t)-\theta_r}{2},\]
	and thus
	\begin{equation}
	\Pr\left[(\theta_{t+1}-\theta_r) n \ge \frac{(\theta_t-\theta_r)n}{2\omega}\;\middle\vert\;~{\cal E}_t,\Gamma_t \right] \ge 1 - \frac{\gamma (2\omega)^{t+1}}{n^{\alpha-1/2}}. \label{lb:case4-bound}
	\end{equation}
	If the event $\neg{\cal E}_t$ occurs, then Fact \ref{sup-no-large-comp} implies that $\theta_{t+1} n = \theta_t n$ with probability $1-O\left(n^{-1}\right)$. Hence, we can remove the conditioning on ${\cal E}_t$ in (\ref{lb:case4-bound}) by adjusting the constant $\gamma$.
	
	Moreover, if $(\theta_{t+1}-\theta_r) n \ge \frac{(\theta_t-\theta_r)n}{2\omega}$ and $\theta_t n > \theta_r n + \frac{n^\alpha}{(2\omega)^t}$, then the event $\Gamma_{t+1}$ occurs. Hence,
	\[\Pr\left[\Gamma_{t+1}\right] \ge \Pr\left[(\theta_{t+1}-\theta_r) n \ge \frac{(\theta_{t+1}-\theta_r)n}{2\omega}\;\middle\vert\;\Gamma_t\right]\Pr\left[\Gamma_t\right] \ge \left( 1 - \frac{\gamma (2\omega)^{t+1}}{n^{\alpha-1/2}}\right) \Pr\left[\Gamma_t\right].\]
	Inducting,
	\[\Pr\left[(\theta_t -\theta_r) n \ge \frac{n^\alpha}{(2\omega)^t}\right] \ge \left(1 - \frac{\gamma (2\omega)^t}{n^{\alpha-1/2}}\right)^t \ge 1 - \frac{\gamma t (2\omega)^t}{n^{\alpha-1/2}}.\]
	Hence, for $t = \frac{1}{20}\log_{2\omega} n$ and $\alpha = 19/20$, $(\theta_t-\theta_r) n \ge n^{9/10}$ w.h.p. The result then follows from Corollary \ref{cor:rc-sub}.
\end{proof}


\section{Local dynamics}\label{sec:comparison}

In this section we prove Theorem \ref{thm:intro3} from the introduction. 

\subsection{Standard background}

Let $P$ be the transition matrix of a finite, ergodic and reversible Markov chain over state space $\Omega$ with stationary distribution $\pi$, and let $1 = \lambda_1 \ge \lambda_2 \ge ... \ge \lambda_n$ denote the eigenvalues of $P$. The {\it spectral gap} of $P$ is defined by $\lambda(P) := 1 - \lambda^*$, where $\lambda^* = \max\{|\lambda_2|,|\lambda_n|\}$. The following bounds on the mixing time are standard (see, e.g., \cite{LPW}):
\begin{equation}\lambda^{-1}(P) - 1 \le \taumix (P) \le \log\left({2e}{\pi_{\rm min}^{-1}}\right)\lambda^{-1}(P)\label{gap-mix},
\end{equation}
where $\pi_{\rm min} = \min_{x \in \Omega} \pi(x)$. 

In this section we will need some elementary notions from functional analysis; for extensive background on the application of such ideas to the analysis of finite Markov chains, see \cite{SC}. If we endow $\R^{|\Omega|}$ with the inner product 
$\langle f,g \rangle_\pi = \sum_{x \in \Omega} f(x)g(x)\pi(x)$,
we obtain a Hilbert space denoted $L_2(\pi) = (\R^{|\Omega|},\langle \cdot,\cdot \rangle_\pi)$. Note that $P$ defines an operator from $L_2(\pi)$ to $L_2(\pi)$ via matrix-vector multiplication. 

Consider two Hilbert spaces $S_1$ and $S_2$ with inner products $\langle\cdot,\cdot\rangle_{S_1}$ and $\langle\cdot,\cdot\rangle_{S_2}$ respectively, and let $R:S_2 \rightarrow S_1$ be a bounded linear operator. The {\it adjoint} of $R$ is the unique operator $R^*:S_1 \rightarrow S_2$ satisfying $\langle f,Rg \rangle_{S_1} = \langle R^*f,g \rangle_{S_2}$ for all $f \in S_1$ and $g \in S_2$. If $S_1 = S_2$, $R$ is {\it self-adjoint} when $R = R^*$. If $R$ is self-adjoint, it is also {\it positive} if $\forall g \in S_2$, $\langle R g, g \rangle_{S_2} \ge 0$. 

\subsection{A comparison technique for Markov chains}\label{subsection:comparison}

Let $H=(V,E)$ be an arbitrary finite graph and let $\Omega_{\rm E} = \{(V,A):~A \subseteq E\}$ be the set of random-cluster configurations on~$H$. Let $P$ be the transition matrix of a finite, ergodic and reversible Markov chain over $\Omega_{\rm E}$ with stationary distribution $\mu=\mu_{p,q}$. For $r \in \N$, let $\Omega_{\rm V} = \{0,1,\ldots,r-1\}^{V}$ be the set of ``$r$-labelings" of $V$, and let $\Omega_{\rm J} = \Omega_{\rm V} \times \Omega_{\rm E}$. Assume $P$ can be decomposed as a product of {\it stochastic} matrices of the form
\begin{equation}
P = M\left(\prod_{i=1}^m T_i \right)M^*, \label{matrix-prod-1}
\end{equation}
where:
\begin{enumerate}[(i)]
	\item $M$ is a $|\Omega_{\rm E}| \times |\Omega_{\rm J}|$ matrix indexed by the elements of $\Omega_{\rm E}$ and $\Omega_{\rm J}$ such that $M(A,(\sigma,B)) \neq 0$ only if $A=B$ for all $A \in \Omega_{\rm E},~(\sigma,B) \in \Omega_{\rm J}$.
	
	\item Each $T_i$ is a $|\Omega_{\rm J}| \times |\Omega_{\rm J}|$ matrix indexed by the elements of $\Omega_{\rm J}$ and reversible w.r.t.\ the distribution $\nu = \mu M$, and such that $T_i((\sigma,A),(\tau,B)) \neq 0$ only if $\sigma=\tau$ for all $(\sigma,A),(\tau,B) \in \Omega_{\rm J}$. 
	\item $M^*$ is a $|\Omega_{\rm J}| \times |\Omega_{\rm E}|$ matrix such that $M^*:L_2(\mu)\rightarrow L_2(\nu)$ is the adjoint of $M:L_2(\nu)\rightarrow L_2(\mu)$.
\end{enumerate}
In words, $M$  assigns a (random) $r$-labeling to the vertices of $H$; $(\prod_{i=1}^m T_i)$ performs a sequence of $m$ operations $T_i$, each of which updates some edges of $H$; and $M^*$ drops the labels from the vertices. These properties imply that $M^*((\sigma,A),B) = \1(A=B)$ and $MM^* = I$. 

Consider now the matrix 
\begin{equation}
P_{\rm L} = M\left(\frac{1}{m}\sum_{i=1}^m T_i \right)M^*.\label{matrix-prod-2}
\end{equation}
It is straightforward to verify that $P_{\rm L}$ is also reversible w.r.t.~$\mu$. 
The following theorem, which generalizes a recent result of Ullrich~\cite{Ullrich1,Ullrich2}, relates
the spectral gaps of $P$ and $P_{\rm L}$ up to a factor of~$O(m\log m)$.


\begin{thm}\label{thm:ullrich-general}
	If $M$, $M^*$ and $T_i$ are stochastic matrices satisfying (i)--(iii) above, and the $T_i$'s are idempotent commuting operators, then
	\[\lambda(P_{\rm L}) \le \lambda(P) \le 8 m \log m \cdot \lambda(P_{\rm L}).\]
\end{thm}
\noindent
We pause to note that this fact has a very attractive intuitive basis.  As noted above, 
$P_{\rm L}$ performs a single update~$T_i$ chosen u.a.r., while $P$ performs all $m$ 
updates~$T_i$, so by coupon collecting one might expect that 
$O(m \log m)$ $P_{\rm L}$ steps should suffice to simulate a single $P$ step. 
However, the proof has to take account of the fact that the $T_i$ updates are interleaved
with the vertex re-labeling operations $M$ and~$M^*$ in $P_{\rm L}$. 
The proofs in~\cite{Ullrich1} and~\cite{Ullrich2} are specific to the case where~$P$
corresponds to the SW dynamics. Our contribution is the realization 
that these proofs still go through (without essential modification) under the more general
assumptions of Theorem~\ref{thm:ullrich-general}, as well as the framework
described above that provides a systematic way of deducing $P_{\rm L}$ from any~$P$
of the form~(\ref{matrix-prod-1}).

Observe that Theorem \ref{thm:ullrich-general} relates the spectral gaps of $P$ and $P_{\rm L}$. We shall see next how to use this technology to obtain mixing time bounds for the heat-bath dynamics using the CM bounds from Sections \ref{sec:upper-bounds} and \ref{sec:lower-bounds}.

\subsection{Application to local dynamics}\label{subsection:local-dynamics}

Let $P_{\rm CM}$ and $P_{\rm HB}$ be the transition matrices of the Chayes-Machta (CM) and heat-bath (HB) dynamics respectively. In this subsection we show that $P_{\rm CM}$ can be expressed as a product of stochastic matrices equivalent to (\ref{matrix-prod-1}) and that $P_{\rm HB}$ is closely related to the corresponding matrix $P_{\rm L}$ in (\ref{matrix-prod-2}). Then, we use Theorem \ref{thm:ullrich-general} to relate the spectral gaps $\lambda(P_{\rm HB})$ and $\lambda(P_{\rm CM})$ and hence prove Theorem \ref{thm:intro3} via (\ref{gap-mix}).

In this case, $\Omega_{\rm V}= \{0,1\}^{V}$ is the set of possible ``active-inactive" labelings of $V$. Consider the $|\Omega_{\rm E}| \times |\Omega_{\rm J}|$ stochastic matrix $M$ defined by
\[M(B,(\sigma,A)) = \1(A=B)\1(A \subseteq E(\sigma)) (q-1)^{f(\sigma,A)}q^{-c(A)},\]
where $E(\sigma) = \left\{(u,v) \in E: \sigma(u)=\sigma(v) \right\}$ and $f(\sigma,A)$ is the number of inactive connected components in $(\sigma,A)$. The adjoint of $M$ is the $|\Omega_{\rm J}| \times |\Omega_{\rm E}|$ stochastic matrix $M^*((\sigma,A),B) = \1(A=B)$. Consider also the family of $|\Omega_{\rm J}| \times |\Omega_{\rm J}|$ stochastic matrices $T_e$ defined for each $e = (u,v) \in E$ as follows:
\[T_e((\sigma,A),(\tau,B)) = \1(\sigma=\tau)\left\{
\begin{array}{cll}
p    & ~~~{\rm if}~B=A \cup e, & \sigma(u) = \sigma(v) = 1; \\
1-p & ~~~{\rm if}~B=A \setminus e, & \sigma(u) = \sigma(v) = 1; \\
1    & ~~~{\rm if}~A(e)=B(e), & \sigma(u) = 0~~{\rm or}~~ \sigma(v) = 0; \\
0 & ~~~{\rm if}~A(e) \neq B(e), & \sigma(u) = 0~~{\rm or}~~ \sigma(v) = 0
\end{array}
\right.
\]
where $\sigma(v) = 1$ (resp., 0) if vertex $v$ is active (resp., inactive) in $\sigma$ and $A(e) = 1$ (resp., $A(e)=0$) if the edge $e$ is present (resp., not present) in $A$.

In words, the matrix $M$ assigns a random active-inactive labeling to a random-cluster configuration, while $M^*$ drops the active-inactive labeling from a joint configuration. The matrix $T_e$ samples $e$ with probability $p$ provided both its endpoints are active. The key observation, which we prove later, is that we can naturally express the CM dynamics as the product of these matrices:

\begin{lemma} \label{lem:representation} $P_{\rm CM} = M\left(\prod\limits_{e \in E} T_e\right)M^*$.
\end{lemma}
\noindent
Now consider the Markov chain given by the matrix 
\[P_{\rm SU} = M\left(\frac{1}{|E|}\sum_{e \in E} T_e \right)M^*,\]
which we call the {\it Single Update (SU) dynamics} and corresponds to the matrix $P_{\rm L}$ defined in (\ref{matrix-prod-2}). Hence, $P_{\rm SU}$ is reversible w.r.t.\ to $\mu = \mu_{p,q}$. Observe that $M$ and $M^*$ clearly satisfy the assumptions of Theorem \ref{thm:ullrich-general}. Moreover, we can easily verify that the $T_e$'s also satisfy these assumptions:

\begin{fact}\label{Te-prop} The $T_e$'s defined above are idempotent commuting operators from $L_2(\nu)$ to $L_2(\nu)$. Moreover, each $T_e$ is reversible w.r.t.~$\nu = \mu M$.
\end{fact}
\begin{proof} The distribution $\nu$ corresponds to the joint Edwards-Sokal measure over $\Omega_{\rm J}$:
	\[\nu(\sigma,A) \propto \left(\frac{p}{1-p}\right)^{|A|} (q-1)^{f(\sigma,A)} \1( A \subseteq E(\sigma))\]
	(see, e.g., \cite{CM}). From this representation, it is straightforward to check that $T_e$ is reversible w.r.t.\ to $\nu$.
	Also, from the definition of $T_e$ it follows that $T_e = T_e^2$ and $T_eT_{e'}=T_{e'}T_e$, which completes the proof.\end{proof}


\noindent
In light of Lemma \ref{lem:representation} and Fact \ref{Te-prop}, we may apply Theorem \ref{thm:ullrich-general} to obtain
\begin{equation}\label{thm5.1bd}
\lambda(P_{\rm SU}) \le \lambda(P_{\rm CM}) \le 8 |E| \log |E| \cdot \lambda(P_{\rm SU}). 
\end{equation}

The SU dynamics is closely related to the HB dynamics. Specifically, their spectral gaps are very similar, as the following fact which we will prove in a moment shows:
\begin{claim}\label{nne} Let $\alpha = (q(1-p)+p)/q^2$; then, \[\alpha \lambda(P_{\rm HB}) \le \lambda(P_{\rm SU}) \le \lambda(P_{\rm HB}).\]
\end{claim}
\noindent
Putting together this claim and~(\ref{thm5.1bd}) yields
\begin{equation}
\alpha \lambda(P_{\rm HB}) \le \lambda(P_{\rm CM}) \le 8 |E| \log |E| \cdot \lambda(P_{\rm HB}), \nonumber
\end{equation}
which relates the spectral gaps of $P_{\rm HB}$ and $P_{\rm CM}$ up to a factor of $\Otilde(n^2)$.
(Note that $\alpha\in[1/q^2,1/q]$, and thus $\alpha = \Theta(1)$.) 
Using (\ref{gap-mix}) this relationship can be translated to the mixing times at the 
cost of a further factor of $\log(\mu_{\rm min}^{-1})$, which is $\Otilde(n^2)$ in the mean-field case.
Theorem~\ref{thm:intro3} now follows immediately from the mixing time bounds on the CM 
dynamics proved in Theorems~\ref{thm:intro1} and~\ref{thm:intro2}.

It remains only for us to supply the missing proofs of Lemma \ref{lem:representation} and Claim \ref{nne}.

\begin{proof}[Proof of Lemma \ref{lem:representation}:] Let ${\cal A}(\sigma) = \left\{(u,v) \in E: \sigma(u)=\sigma(v)=1 \right\}$ and let ${\cal T} = \prod\limits_{e \in E} T_e$. Observe that
	\[{\cal T}((\sigma,A),(\tau,B))= \1(\sigma=\tau)\1(A \setminus {\cal A}(\sigma)= B \setminus {\cal A}(\sigma)) p^{|{\cal A}(\sigma) \cap B|}(1-p)^{|{\cal A}(\sigma)|-|{\cal A}(\sigma)\cap B|}.\]
	Then, from the definitions of $M$, $M^*$ and ${\cal T}$, we obtain
	\begin{eqnarray}
	& & M{\cal T}M^*(A,B) = \sum\limits_{(\sigma,C)}\sum\limits_{(\tau,D)} M(A,(\sigma,C)){\cal T}((\sigma,C),(\tau,D))M^*((\tau,D),B) \nonumber\\
	&=& \sum\limits_{\sigma \in \Omega_{\rm V}} \1(A \subseteq E(\sigma)) (q-1)^{f(\sigma,A)}q^{-c(A)} \1(A \setminus {\cal A}(\sigma)= B \setminus {\cal A}(\sigma)) p^{|{\cal A}(\sigma) \cap B|}(1-p)^{|{\cal A}(\sigma)|-|{\cal A}(\sigma)\cap B|}.\nonumber
	\end{eqnarray}
	
	Now observe that if $A \subseteq E(\sigma)$, then sub-step {\rm (i)} of the CM dynamics chooses $\sigma \in \Omega_{\rm V}$ with probability $(q-1)^{f(\sigma,A)}q^{-c(A)}$. Moreover, if after sub-step {\rm (i)} the joint configuration obtained is $(\sigma,A)$, then the probability of obtaining $B$ in sub-step {\rm (iii)} is $p^{|{\cal A}(\sigma) \cap B|}(1-p)^{|{\cal A}(\sigma)|-|{\cal A}(\sigma)\cap B|}$ provided $A$ and $B$ differ only in the active part of the configuration ${\cal A}(\sigma)$. Thus, $M{\cal T}M^*(A,B) = P_{\rm CM}(A,B)$.
\end{proof}

\begin{proof}[Proof of Claim \ref{nne}:] First we show that $P_{\rm SU}$ is positive. This was already shown for $P_{\rm HB}$ in \cite[Lemma 2.7]{Ullrich3}. Recall each $T_e$ is reversible w.r.t.\ to $\nu$, and thus the operator $T_e:L_2(\nu) \rightarrow L_2(\nu)$ is self-adjoint (see, e.g., \cite{SC}). Since $T_e$ is self-adjoint and idempotent, it is also positive for every $e \in E$ (see, e.g. \cite[Thm. 9.5-1 \& 9.5-2]{K}). Therefore, for $f \in \R^{|\Omega_{\rm E}|}$ we have
	\[\langle P_{\rm SU} f , f \rangle_\mu = \frac{1}{|E|} \sum_{e \in E}\langle MT_eM^* f , f \rangle_\mu = \frac{1}{|E|} \sum_{e \in E}\langle T_eM^* f , M^*f \rangle_\nu \ge 0,\]
	and thus $P_{\rm SU}$ is positive.
	
	Given a random-cluster configuration $(V,A)$, it is straightforward to check that one step of
	the SU dynamics is equivalent to the following discrete steps:
	\begin{enumerate}[(i)]
		\itemsep0em
		\item {\it activate\/} each connected component of $(V,A)$ independently with probability $1/q$;
		\item pick $e \in E$ u.a.r.;
		\item if both endpoints of $e$ are active, add $e$ with probability $p$ and remove it otherwise. (If either endpoint of $e$ is inactive, do nothing.)
	\end{enumerate}
	Similarly, recalling the definition of the HB dynamics from the introduction, it is easy to check that
	each step is equivalent to the following:
	\begin{enumerate}[(i)]
		\itemsep0em
		\item[(i)] pick an edge $e\in E$ u.a.r.;
		\item[(ii)] include the edge~$e$ in the new configuration with probability $p_e$, where $$
		p_e = \begin{cases}
		{\textstyle\frac{p}{p+q(1-p)}} & \text{if $e$ is a cut edge in $(V,A\cup \{e\})$;}\\
		p & \text{otherwise.}
		\end{cases} $$
		(The rest of the configuration is left unchanged.)
	\end{enumerate}
	Note that $e$ is a cut edge in $(V,A\cup \{e\})$ iff changing the current configuration
	of~$e$ changes the number of connected components.
	
	Using these definitions for the SU and HB dynamics, it is an easy exercise to check that for $A \neq B$ and $\alpha = (q(1-p)+p)/q^2$,
	\[\alpha P_{\rm HB}(A,B) \le P_{\rm SU}(A,B) \le P_{\rm HB}(A,B).\] 
	Since $P_{\rm SU}$ is positive, the result follows from Lemma 2.5 in \cite{Ullrich3}.
\end{proof}

\bibliography{mybib}{}
\bibliographystyle{plain} 

\end{document}